\documentclass{amsart}
\usepackage{amssymb,amsmath,enumerate}
\usepackage[curve]{xypic}

\title[Characterizations of categories of commutative C*-subalgebras]{Characterizations of categories of \\commutative C*-subalgebras}
\author{Chris Heunen}
\address{University of Oxford}
\email{heunen@cs.ox.ac.uk}
\date{\today}
\subjclass{46L35, 18F99, 06B75, 81P10}
\keywords{C*-algebra, maximal abelian subalgebra, injective
  $*$-homomorphism, category, Grothendieck construction}
\thanks{The author is grateful to Jonathon Funk for helping him
  understand and in fact suggesting large parts of
  Appendix~\ref{sec:inversesemigroups}, and 
  to Manny Reyes, for many interesting discussions, especially about
  functoriality. 
  The author was supported by the Netherlands
  Organisation for Scientific Research and the Office of Naval Research.} 

\newcommand{\cat}[1]{\ensuremath{\mathbf{#1}}}
\newcommand{\Cat}[1]{\ensuremath{\mathbf{#1}}}
\newcommand{\op}{\ensuremath{^{\mathrm{op}}}}
\newcommand{\after}{\ensuremath{\circ}}
\newcommand{\id}[1][]{\ensuremath{\mathrm{id}_{#1}}}
\newcommand{\PSh}{\mathrm{PSh}}
\newcommand{\B}{\ensuremath{\mathcal{B}}}
\newcommand{\C}{\ensuremath{\mathcal{C}}}
\newcommand{\Cs}{\ensuremath{\C_{\subseteq}}}
\newcommand{\Cm}{\ensuremath{\C_{\rightarrowtail}}}
\newcommand{\Ci}{\ensuremath{\C_{\cong}}}
\newcommand{\V}{\ensuremath{\mathcal{V}}}
\newcommand{\Vs}{\ensuremath{\V_{\subseteq}}}
\newcommand{\Vm}{\ensuremath{\V_{\rightarrowtail}}}
\newcommand{\Aut}{\ensuremath{\mathrm{Aut}}}
\newcommand{\field}[1]{\ensuremath{\mathbb{#1}}}
\newcommand{\ie}{\textit{i.e.}~}

\DeclareMathOperator{\Spec}{\mathrm{Spec}}
\DeclareMathOperator{\Proj}{\mathrm{Proj}}
\DeclareMathOperator{\dom}{\mathrm{dom}}
\DeclareMathOperator{\cod}{\mathrm{cod}}
\DeclareMathOperator{\rk}{\mathrm{rank}}

\theoremstyle{plain}
\newtheorem{theorem}{Theorem}[section]
\newtheorem{corollary}[theorem]{Corollary}
\newtheorem{lemma}[theorem]{Lemma}
\newtheorem{proposition}[theorem]{Proposition}

\theoremstyle{definition}
\newtheorem{definition}[theorem]{Definition}
\newtheorem{example}[theorem]{Example}
\newtheorem{remark}[theorem]{Remark}

\begin{document}
\maketitle

\begin{abstract}
  We aim to characterize the category of injective $*$-homomorphisms
  between commutative C*-subalgebras of a given C*-algebra $A$. We reduce
  this problem to finding a weakly terminal commutative
  subalgebra of $A$, and solve the latter for various C*-algebras, including
  all commutative ones and all type I von Neumann algebras.
  This addresses a natural generalization of the Mackey--Piron programme: which
  lattices are those of closed subspaces of Hilbert space?
  We also discuss the way this categorified generalization differs from the original
  question. 
\end{abstract}

\section{Introduction}

The collection $\C(A)$ of commutative C*-subalgebras of a fixed 
C*-algebra $A$ can be made into a category under various choices of
morphisms. Two natural ones are inclusions and injective $*$-homomorphisms,
resulting in categories $\Cs(A)$ and $\Cm(A)$, respectively. The goal
of this article is to characterize these categories.

Categories based on $\C(A)$ are interesting for a number of reasons.
A first motivation to study such categories is the hope
that they could lead to a noncommutative extension of Gelfand duality.
It is known that $\Cs(A)$ determines $A$ as a partial
C*-algebra~\cite{vdbergheunen:colim}. 
Except when $A\cong \mathbb{C}^2$ or $A \cong \mathbb{M}_2(\mathbb{C})$,
equivalently $\Cs(A)$
determines precisely the quasi-Jordan structure of $A$~\cite{hamhalter:pseudojordan,hamhalterturilova:pseudojordan}.
Thus, $\C(A)$ in itself is already an interesting invariant of
$A$. Moreover, structures based on $\C(A)$ circumvent obstructions to
a noncommutative Gelfand duality that afflict many
other candidates~\cite{vdbergheunen:nogo}.  
Indeed, for C*-algebras $A$ with enough projections, adding a little more
structure to $\C(A)$ fully determines the algebra structure of
$A$~\cite{heunenreyes:awstar,heunenreyes:diagonal}. To get a full noncommutative Gelfand
duality for such algebras, it suffices to characterize the
structures based on $\C(A)$ that arise this way; an important step is
clearly to characterize categories of the form $\C(A)$. 

Second, there is a physical perspective on $\C(A)$. The underlying
idea, due to Bohr, is that one can only empirically access a quantum mechanical
system, whose observables are modeled by a (noncommutative)
C*-algebra, through its classical subsystems, as modeled by commutative C*-subalgebras~\cite{heunenetal:topos}.
Categories based on $\C(A)$ are of paramount importance in the
recent uses of topos theory in research in foundations of physics based
on this idea that proposes a new form of quantum logic~\cite{doeringisham:topos,heunenetal:bohrification}. 
Knowing which categories are of the form $\C(A)$ also characterizes
which toposes are of the form studied in that programme. This should
increase insight into the intrinsic structure of such toposes, and
hence shed light on the foundations of  quantum physics such toposes
aim to describe logically.  

Third, more generally, a characterization of $\C(A)$ satisfactorily addresses a
general theme in research in foundations of quantum
mechanics. For example, it addresses (a categorification of)
the Mackey--Piron programme. This programme asks the question: which
orthomodular lattices are those of closed subspaces of Hilbert
space? 
(See~\cite{piron:foundations,soler:orthomodular,kalmbach:measures}.) 
A characterization of $\C(A)$ would provide an answer, because choosing a
commutative C*-subalgebra of the matrix algebra $\mathbb{M}_n(\field{C})$
amounts to choosing an orthonormal subset and hence a closed subspace
of $\field{C}^n$, and an
appropriate generalization to infinite dimension holds as well  (see
also Theorem~\ref{VsProj} below and~\cite{heunen:complementarity}). 
Similarly, a characterization of $\C(A)$ has consequences in the study
of test spaces. These are defined as collections of orthogonal
subsets of a Hilbert space satisfying some conditions, and have been
proposed as axioms for operational quantum mechanics. One of the major
questions there is again which test spaces arise from propositions on Hilbert
spaces~\cite{wilce:testspaces}. 

Our main result is to reduce characterizing $\Cm(A)$ to
finding a weakly terminal commutative subalgebra of $A$. 
This is closely related to analyzing all maximal abelian subalgebras
(masas). Explicating the structure of masas of C*-algebras in general is a
hard problem, and not much seems to be known systematically outside of
the case of factors of type I and type $\mathrm{II}_1$;
see~\cite{bures:masas,sinclairsmith:masas}. Fortunately, finding 
a weakly terminal commutative subalgebra is generally easier than
finding all masas. We prove that the following classes of C*-algebras $A$
possess weakly terminal commutative subalgebras, and therefore we find a
full characterization of $\Cm(A)$ for:
\begin{itemize}
\item type I von Neumann algebras, including all
  finite-dimensional C*-algebras;
\item commutative C*-algebras.
\end{itemize}

The strategy behind our characterization is as follows. The key
insight is to recognize $\Cm(D)$ for a commutative C*-algebra $D$ as
the Grothendieck construction of an action of a monoid $M$ on a
partially ordered set $P$. We characterize such so-called
amalgamations. Next, we use known results to characterize the
partially ordered set $P=\Cs(D)$, consisting of 
partitions of the Gelfand spectrum of $D$. Then, we show that $\Cm(A)$ is
equivalent to $\Cm(D)$ for a weakly terminal object $D$ in $\Cm(A)$.
Finally, we establish such a weakly terminal object $D$ for the
various types of C*-algebras $A$ mentioned, finishing the characterization. 
This last step is the only one limiting our characterization to
C*-algebras $A$ with weakly terminal commutative subalgebras. 
Summarizing:
\begin{enumerate}[(1)]
  \item show that a C*-algebra $A$ has a weakly terminal abelian
    subalgebra $D$;
  \item show that $\Cm(A)$ is equivalent to $\Cm(D)$;
  \item show that $\Cm(D)$ is equivalent to $P(X) \rtimes S(X)$, with $X$ 
    the spectrum of $D$;
  \item characterize $P(X) \rtimes S(X)$ in terms of $P(X)$ and $S(X)$;
  \item a characterization of $P(X)$ exists;
  \item in the cases in question, $X$, and hence $S(X)$, is easy to characterize.
\end{enumerate}

Thus we address the Mackey--Piron programme in a different way than the theorems of Piron~\cite{piron:foundations} and Sol{\`e}r~\cite{soler:orthomodular}, which together form the only characterization of the lattice of closed subspaces of a Hilbert space we are aware of. 
Piron's theorem states that the lattice should be complete, atomic, irreducible, orthomodular, and satisfy the covering law, from which it follows that it must be the lattice of closed subspaces of some Euclidean space over a skew field. Sol{\`e}r's theorem says that if additionally this Euclidean space is infinite-dimensional and has the property that any closed subspace is a direct summand, then the skew field must be the reals, complexes or quaternions, and the space must be a Hilbert space.
Both Sol{\`e}r's direct summand condition and Piron's lattice-theoretic axioms relate to our use of partition lattices $P(X)$, but instead of orthomodularity we use the action of $S(X)$.
Interestingly, our results apply to arbitrary Hilbert spaces, whereas Sol{\`e}r's theorem only holds for infinite-dimensional ones. 

The paper is structured as follows. We start with
Section~\ref{sec:preliminaries}, which introduces the poset $\Cs(A)$ and the
category $\Cm(A)$ and discusses their basic properties and motivation. A more
in-depth analysis of the relationship between the two, again depending
on the Grothendieck construction, is made later, in
Section~\ref{sec:inclusionsinjections}. Our main results
are presented in between. To aid intuition, we first cover the finite-dimensional
case, and only then incorporate the subtleties of the
infinite-dimensional case. Section~\ref{sec:groupamalgamations}
characterizes amalgamations of groups and posets, which is then used
in Section~\ref{sec:finite} to establish the characterization in the
finite-dimensional case. Then, Section~\ref{sec:monoidamalgamations}
refines the earlier analysis to characterize amalgamations of monoids
and posets. This is used in Section~\ref{sec:infinite} to establish
the characterization in the infinite-dimensional case.
Appendix~\ref{sec:inversesemigroups} records some intermediate results
of independent interest. In particular, it discusses an alternative
way to investigate the relationship between $\Cm(A)$ and $\Cs(A)$.
 
To end this introduction let us briefly indicate the differences between $\Cm(A)$ and $\Cs(A)$.
This will be discussed in more depth in Section~\ref{sec:inclusionsinjections}, but it might be helpful to mention them now to set the scene. Any morphism in $\Cm(A)$ factors uniquely as a $*$-isomorphism followed by a morphism in $\Cs(A)$.
If $\Cm(A) \cong \Cm(B)$ are isomorphic categories, then $\Cs(A) \cong \Cs(B)$ are isomorphic posets. Therefore, as discussed above, both categories $\Cm(A)$ and $\Cs(A)$ are \emph{invariants} of the C*-algebra $A$, in the sense that both determine the (quasi-)Jordan structure of $A$, and are hence respected by (quasi-)Jordan homomorphisms. 
We will mostly be interested in a coarser notion of invariant, namely equivalence of categories, rather than isomorphism of categories. For posetal categories like $\Cs(A)$, isomorphism and equivalence coincide, but for $\Cm(A)$ this makes a difference: $\Cm(A) \simeq \Cm(B)$ need not imply $\Cs(A) \cong \Cs(B)$ (and certainly not $A \cong B$). It turns out that $\Cm(A) \simeq \Cm(B)$ are equivalent categories precisely when $\Cs(A)$ and $\Cs(B)$ are Morita-equivalent, in the sense that they have equivalent presheaf categories $\PSh(\Cs(A)) \simeq \PSh(\Cs(B))$. This explains why equivalence of categories is a more natural invariant from the point of view of category theory and topos theory.

\section{Motivation}\label{sec:preliminaries}

We do not require C*-algebras to have a unit, and write $\cat{Cstar}$ for the category of C*-algebras and $*$-homomorphisms.

\begin{definition}
  Write $\C(A)$, or simply $\C$, for the collection of nonzero commutative
  C*-subalgebras $C$ of a C*-algebra $A$. This set of objects can be made
  into a category by various choices of morphisms, such as:
  \begin{itemize}
  \item inclusions $C \hookrightarrow C'$, given by $c \mapsto c$,
    yielding a (posetal) category $\Cs(A)$;
  \item injective $*$-morphisms $C \rightarrowtail C'$, giving a
    (left-cancellative) category $\Cm(A)$.
  \end{itemize}
\end{definition}

These two categories are interesting for two related reasons. First, they
form a major ingredient in a new attack on a noncommutative extension of Gelfand
duality~\cite{vdbergheunen:colim,vdbergheunen:nogo,heunenreyes:awstar}. Essentially, one
could think of them as invariants of a C*-algebra.
Second, they play an important role in the recent use of topos theory
in the foundations of quantum physics. From this perspective, one
could think of them as encoding the logic of a quantum-mechanical
system whose observables are modeled by the C*-algebra $A$.
We will discuss these two perspectives in turn, but first we consider
functoriality of the construction $A \mapsto \C(A)$.
Section~\ref{sec:inclusionsinjections} below discusses the
relationship between the two choices of morphisms, $\Cm(A)$ or
$\Cs(A)$ in more detail.  

\subsection*{Functoriality}

The assignment $A \mapsto \Cs(A)$ extends to a functor: given a
$*$-homomorphism $\varphi \colon A \to B$, direct images $C \mapsto \varphi(C)$ form
a morphism $\Cs(A) \to \Cs(B)$ of posets, for if $C \subseteq C'$,
then $\varphi(C) \subseteq \varphi(C')$. Well-definedness relies on
the following fundamental fact, that we record as a lemma for future reference. 

\begin{lemma}\label{cstarimage}
  The set-theoretic image of a C*-algebra under a $*$-homomorphism is
  again a C*-algebra. 
\end{lemma}
\begin{proof}
  See~\cite[Theorem~4.1.9]{kadisonringrose:operatoralgebras}.
  \qed
\end{proof}

The assignment $A \mapsto \Cm(A)$ has to be adapted to be made
functorial. Either we only consider injective $*$-homomorphisms $A
\rightarrowtail B$, or we restrict the target category $\Cm(A)$ as follows.
Write $\Cat{Cat}$ for the category of small categories and functors.

\begin{lemma}\label{Cmfunctorial}
  There is a functor $\Cat{Cstar} \to \Cat{Cat}$, sending $A$ to the
  subcategory of $\Cm(A)$ with morphisms those $i \colon C \rightarrow
  C'$ satisfying
  \begin{equation*}\label{injectionextends}
    i^{-1}(I \cap C') = I \cap C
  \end{equation*}
  for all closed (two-sided) ideals $I$ of $A$.
\end{lemma}
\begin{proof}
  Let $\varphi \colon A \to B$ be a $*$-homomorphism, and let $i$ be as
  in the statement of the lemma. Then $i$ induces a well-defined
  injective $*$-homomorphism $\varphi(C) \to 
  \varphi(C')$ precisely when $\varphi(c_1)=\varphi(c_2) \Longleftrightarrow
  \varphi(i(c_1)) = \varphi(i(c_2))$. Since $\varphi$ and $i$ are linear, this
  comes down to $\varphi(c)=0 \Longleftrightarrow \varphi(i(c))=0$, \ie
  $\ker(\varphi) \cap C = \ker(\varphi \after i)$. 
  Setting $I=\ker(\varphi)$, this becomes
  \begin{align*}
    I \cap C
    & = \{ c \in C \mid \varphi(c)=0 \} \\
    & = \{ c \in C \mid \varphi(i(c))=0\} \\
    & = i^{-1}\big( \{ c' \in C' \mid \varphi(c') = 0 \} \big) \\
    & = i^{-1}(I \cap C')
  \end{align*}
  and is therefore satisfied.
  \qed
\end{proof}

Notice that $*$-homomorphisms satisfying the condition of the previous
lemma are automatically injective, as is seen by taking $I=\{0\}$.

Notice also that when $A$ is a topologically simple C*-algebra, such as the
algebra $\mathbb{M}_n(\mathbb{C})$ of $n$-by-$n$ complex matrices, then the subcategory of the 
previous lemma is actually the whole category $\Cm(A)$.



\subsection*{Invariants}

Let us temporarily consider von Neumann algebras $A$ and their von
Neumann subalgebras $\V(A)$, giving categories $\Vs$ and $\Vm$. We
will show that $\Vs$ contains exactly the same information as the
lattice $\Proj(A)$ of projections of $A$, in the technical sense that they are functors with equivalent images. This lattice has been
studied in depth, so from the point of view of (new) invariants of $A$, the
category $\Vm$ is more interesting.
See also Remark~\ref{invariants} below. By extension, $\Cm$ is
possibly more interesting as an invariant than $\Cs$, because
$\C(A)$ and $\V(A)$ coincide for finite-dimensional C*-algebras $A$. 

Denote the category of von Neumann algebras and unital normal
$*$-homo\-mor\-phisms by $\Cat{Neumann}$, and write $\Cat{cNeumann}$ for
the full subcategory of commutative (unital von Neumann) algebras. Denote
the category of orthomodular lattices and lattice morphisms preserving the
orthocomplement by $\Cat{Ortho}$. The functor $\Proj \colon
\Cat{Neumann} \to \Cat{Ortho}$ takes $A$ to $\{p \in A \mid
p^2=p=p^*\}$ under the ordering $p \leq q$ iff $pq=p$. On morphisms
$f \colon A \to B$ it acts as $p \mapsto f(p)$. 
Recall that the essential image of a functor $F$ is the smallest subcategory
of the target category containing all isomorphisms and all morphisms of the form $F(f)$.
Denote the essential
image of $\Proj$ by $\cat{D}$; traditional quantum logic is the study
of this subcategory of $\Cat{Ortho}$~\cite{redei:quantumlogic}. 

Denote by $\Cat{Poset}[\Cat{cNeumann}]$ the following category: objects are
sets of commutative von Neumann algebras partially ordered by
inclusion (\ie $C \leq C'$ iff $C \subseteq C'$); morphisms
are monotonic functions. We may regard 
$\Vs$ as a functor $\Cat{Neumann} \to
\Cat{Poset}[\Cat{cNeumann}]$. Denote the essential image of $\Vs$ by
$\cat{C}$; this is a subcategory of $\Cat{Poset}[\Cat{cNeumann}]$. 

We now define two new functors, $F \colon \cat{C} \to \cat{D}$ and $G
\colon \cat{D} \to \cat{C}$. The functor $F$ acts on an object $\Vs(A)$ as
follows. For each $C \in \Vs(A)$, we know that $\Proj(C)$ is a Boolean
algebra~\cite[4.16]{redei:quantumlogic}. Because additionally the
hypothesis of Kalmbach's Bundle lemma, is
satisfied, these Boolean algebras unite into an orthomodular lattice
$F(\Vs(A))$. This assignment extends naturally to morphisms. 

\begin{lemma}[Bundle lemma]
  Let $\{B_i\}$ be a family of Boolean algebras such that $\vee_i=\vee_j$, $\neg_i=\neg_j$,
  and $0_i=0_j$ on intersections $B_i \cap B_j$. If $\leq$ on
  $\bigcup_i B_i$ is transitive and makes it into a lattice,
  then $\bigcup_i B_i$ is an orthomodular lattice.
\end{lemma}
\begin{proof}
  See~\cite[1.4.22]{kalmbach:orthomodularlattices}.  
  \qed
\end{proof}

The functor $G$ acts on the projection lattice $L$ of a von Neumann
algebra as follows. Consider all complete Boolean sublattices $B$ of $L$ as a
poset under inclusion. For each $B$, the continuous functions on its
Stone spectrum form a commutative von Neumann algebra. Thus we obtain
an object $G(L)$ in $\cat{C}$, and this assignment extends naturally to morphisms. 


\begin{theorem}\label{VsProj}
  The functors $F$ and $G$ form an equivalence, and make the following diagram commute.
  \[\xymatrix@R-2ex{
    & \Cat{Neumann} \ar_-{\Vs}[dl] \ar^-{\Proj}[dr] \\
    \cat{C} \ar@<.25ex>^-{F}[rr] \ar@<-.33ex>@{}|-{\simeq}[rr] 
    && \cat{D} \ar@<1ex>^-{G}[ll]
  }\]
\end{theorem}
\begin{proof}
  Follows directly from the definitions and the previous lemma.
  \qed
\end{proof}

Indeed, both $\Vs(A)$ and $\Proj(A)$ capture the Jordan algebra structure of $A$~\cite{hardingdoering:jordan}, excepting the case where $A$ has summands of type $I_2$. 

Returning to the setting of C*-algebras, notice that the previous theorem
fails, because there are C*-algebras without any nontrivial
projections. But every C*-algebra has many commutative 
C*-subalgebras: every self-adjoint element generates one, and every element of a
C*-algebra is a complex linear combination of self-adjoint
elements. For C*-algebras, $\Cs(A)$ captures precisely the
pseudo-Jordan algebra structure of $A$~\cite{hamhalter:pseudojordan,hamhalterturilova:pseudojordan}.
In this regard, it is also worth remarking that the functor $\Cs \colon\Cat{Cstar} \to \Cat{Poset}[\Cat{cCstar}]$
factors through the category of partial C*-algebras~\cite{vdbergheunen:colim}.

\subsection*{Toposes in foundations of physics}

The main theorem in the application of topos theory to foundations of
quantum physics is the following. The tautological functor $C \mapsto C$ is an
internal (possibly nonunital)
C*-algebra~\cite[Theorem~6.4.8]{heunenetal:bohrification}. It holds in
both toposes $\Cat{Set}^{\Cs}$ and $\Cat{Set}^{\Cm}$ because 
of the fundamental Lemma~\ref{cstarimage} above. Categorically, $\Cm$
is a more natural choice than $\Cs$. 

But to characterize a presheaf category is the same as characterizing
the category it is based on, by Morita equivalence; see also
Section~\ref{sec:inclusionsinjections} and
Appendix~\ref{sec:inversesemigroups} below. Thus, our main results
also characterize toposes of the form $\Cat{Set}^{\Cm}$. 
For a more or less practical account of the above
folklore knowledge we refer to~\cite{bunge:presheaves}.


\section{Poset-group-amalgamations}\label{sec:groupamalgamations}

This section recalls the Grothendieck construction, focusing on the
special case of an action of a group on a poset. We will call the
resulting categories poset-group-amalgamations. The goal of this
section is to characterize such categories. This is interesting in its
own right, but even more so because in Section~\ref{sec:infinite} we will 
see that $\Cm$ is of this form. For that reason, we
prefer a practical characterization. Therefore, we will not pursue the
highest possible level of generality: the discussion in this
section is in elementary terms spelling out what is probably folklore
knowledge. In particular, the characterization in this section can
be extended to poset-category-amalgamations, and perhaps even to a
characterization of Grothendieck constructions of arbitrary indexed
categories, but we will not pursue this here.
We will use the Grothendieck construct, also called the
category of elements, again in Section~\ref{sec:inclusionsinjections},
where it is discussed more abstractly. The main idea in this section
is to factor out symmetries into a monoid action, leaving just the
partial order. 

\begin{definition}\label{action}
  An \emph{action} of a monoid $M$ (in the category of sets) \emph{on} a category $\cat{C}$ is
  a functor $F \colon M \to \Cat{Cat}(\cat{C},\cat{C})$. Write $mx$
  for the action of $Fm$ on an object $x$ of $\cat{C}$, and $mf$ for
  the action of $Fm$ on a morphism $f$ of $\cat{C}$. 
\end{definition}

\begin{definition}
  If a monoid $M$ acts on a category $\cat{C}$, then we can perform
  the \emph{Grothendieck construction}: we can make a new
  category $\cat{C} \rtimes M$ whose objects are those 
  of $\cat{C}$, and whose morphisms $x \to y$ are pairs $(m,f)$ such
  that $\dom(f) = x$ and $\cod(f)=my$. Composition and identities are
  inherited from $M$ and $\cat{C}$. Explicitly, $\id[x] = (1,\id[x])$,
  and $(n,g) \circ (m,f) = (mn,(mg)f)$.
\end{definition}

If the category $\cat{C}$ in the previous definition is
a partially ordered set $P$, then $P \rtimes M$ has as objects $p \in P$, and
morphisms $p \to q$ are $m \in M$ such that $p \leq mq$, with unit
and composition from $M\op$.

An illustrative example to keep in mind is the following. Let $M$ be
the group of unitary $n$-by-$n$ matrices. Let $P$ be the lattice of
subspaces of $\field{C}^n$, ordered by inclusion. Then $M$
acts on $P$ by $UV = \{ U(v) \mid v \in V \}$ for $U \in M$ and $V \in
P$. Morphisms in $P \rtimes M$ between subspaces $V \subseteq 
\field{C}^n$ and $W \subseteq \field{C}^n$ are unitary matrices $U$ 
such that $U^{-1}(v) \in W$ for all $v \in V$.

This section characterizes categories of the form $P \rtimes G$ for an
action of a group $G$ on a poset $P$ with a least
element. Our characterization will rely on weakly initial objects to
recover $P$ from $P \rtimes G$. Categorically, this is trivial, but as
we will see in Sections~\ref{sec:finite} and~\ref{sec:infinite}, it is a very
important step in our application.
An object $0$ is \emph{weakly initial} when for any
object $x$ there exists a (not necessarily unique) morphism $0 \to
x$; notice that such an object is not necessarily unique up to isomorphism, as an initial object would be. If a category $\cat{A}$ has a weak initial object $0$, we can
regard the endohomset monoid $\cat{A}(0,0)$ as a one-object category. 
Recall that a \emph{retraction} of a functor is a left-inverse. 

\begin{lemma}
  If a category $\cat{A}$ has a weak initial object $0$ and a faithful
  retraction $F$ of the inclusion $\cat{A}(0,0) \hookrightarrow \cat{A}$, then its
  objects are preordered by 
  \[
  x \leq y \iff \exists f \in \cat{A}(x,y) .\, F(f)=1.
  \]
\end{lemma}
\begin{proof}
  Clearly $\leq$ is reflexive, because $F(\id[x])=1$.  It is also
  transitive, for if $x \leq y$ and $y \leq z$, then there are $f
  \colon x \to y$ and $g \colon y \to z$ with $F(f)=1=F(g)$, so that
  $g \after f \colon x \to z$ satisfies $F(g \after f) = F(g) \after
  F(f) = 1 \after 1 = 1$ and $x \leq z$.
  \qed
\end{proof}

Thus we can recover the group $G$ from $\cat{A}=P \rtimes G$ by
looking at $\cat{A}(0,0)$. We can also recover the poset $P$ from
$\cat{A}$ by the previous lemma. What is left is to reconstruct the action
of $G$ on $P$ given just $\cat{A}$. For $m \in G$ and $p \in P$, we
can access the object $mq$ through the morphisms $m \colon p \to q$ in
$\cat{A}$. There is always at least one such morphism, namely $m
\colon mq \to q$, because trivially $mq \leq mq$. In fact, this is
always an isomorphism. We will now use this fact to recover the action of
$G$ on $P$ from $\cat{A}$. We will call this an \emph{amalgamation} by analogy with the use of the term in algebra.

\begin{definition}
  A category $\cat{A}$ is called a \emph{poset-group-amalgamation}
  when there exist a partial order $P$ and a group $G$ such that: 
  \begin{enumerate}
  \item[(A1)] there is a weak initial object $0$, unique up to isomorphism;
  \item[(A2)] there is a faithful retraction $F$ of the inclusion
    $\cat{A}(0,0) \hookrightarrow \cat{A}$;
  \item[(A3)] there is an isomorphism $\alpha \colon \cat{A}(0,0) \to G\op$ of
    monoids;
  \item[(A4)] there is an equivalence $\smash{\xymatrix@1{(\cat{A},\leq)
        \ar@<1ex>|-{\beta}[r] & P \ar@<.5ex>|-{\beta'}[l]}}$ of preorders;
  \item[(A5)] for each object $x$ there is an isomorphism $f \colon x
    \to \beta'(\beta(x))$ with $\alpha F(f)=1$;
  \item[(A6)] for each $y$ and $m$ there is an isomorphism $f \colon x \to y$ with $\alpha F(f)=m$.
  \end{enumerate}
\end{definition}

\begin{example}
  If $P$ is a partial order with least element, and $G$ is a
  group acting on $P$, then $P \rtimes G$ satisfies (A1)--(A6).
\end{example}
\begin{proof}
  The least element $0$ of $P$ is a weak initial object, satisfying
  (A1). Conditions (A2)--(A4) are satisfied by definition, and (A5) is vacuous.  
  To verify (A6) for $q \in P$ and $m \in G$, notice that $mq \leq
  mq$, so $f=1\colon mq \to mq$ is an isomorphism with $\alpha F(f)=1$.
  \qed
\end{proof}

\begin{lemma}\label{recoverGaction}
  If $\cat{A}$ satisfies (A1)--(A6), then it induces an action of $G$
  on $P$ given by $mp = \beta(x)$ if $f \colon x \to 
  \beta'(p)$ is an isomorphism with $\alpha(F(f))=m$.
\end{lemma}
\begin{proof}
  First, notice that for any $p \in P$ and $m \in G$ there exists an
  isomorphism $f \colon x \to \beta'(p)$ with $\alpha(F(f))=m$ by
  (A6). If there is another isomorphism $f' \colon x' \to \beta'(p)$
  with $\alpha(F(f'))=m$, then their composition gives $x \cong x'$,
  and therefore $\beta(x) \cong \beta(x')$. But because $P$ is a
  partial order, this means $\beta(x)=\beta(x')$. Thus the action is
  well-defined on objects. 

  To see that it is well-defined on morphisms, suppose that $p \leq
  q$. Then there is a morphism $f \colon \beta'(p) \to \beta'(q)$ with
  $F(f)=1$. For any $m \colon 0 \to 0$, axiom (A6) provides 
  isomorphisms $f_p \colon x_p \to \beta'(p)$ and $f_q \colon x_q \to
  \beta'(q)$ with $\alpha(F(f_p)) = m = \alpha(F(f_q))$. Then
  $f=f_q^{-1} f f_p \colon x_p \to x_q$ is an isomorphism satisfying
  $\alpha F(f) = m m^{-1} = 1$. So $mp \leq mq$.

  Next, we verify that this assignment is functorial $G \to
  \Cat{Cat}(P,P)$. Clearly $\id[\beta'(p)]$ is an isomorphism
  $x \to \beta'(p)$ with $F(\id[\beta'(p)])=1$. Therefore $1p=\beta(\beta'(p))=p$.

  Finally, for $m_2,m_1 \in M$ and $p \in P$, we have $m_1p=\beta(x_1)$ where $f_1
  \colon x_1 \to \beta'(p)$ is an isomorphism with $\alpha(F(f_1))=m_1$. So
  $m_2(m_1p) = \beta(x_2)$ where $f_2 \colon x_2 \to
  \beta'(\beta(x_1))$ is an isomorphism with $\alpha(F(f_2))=m_2$. By (A5),
  there is an isomorphism $h \colon x_1 \to \beta'(\beta(x_1))$ with
  $F(h)=1$. So $f = f_1 h^{-1} f_2$ is an isomorphism $x_2 \to
  \beta'(p)$ with $\alpha(F(f)) = m_2 m_1$. Thus
  $(m_2m_1)p=\beta(x_2)=m_2(m_1p)$. 
  \qed
\end{proof}

\begin{theorem}\label{charPG}
  If $\cat{A}$ satisfies (A1)--(A6), then there is an equivalence
  $\cat{A} \to P \rtimes G$ given by $x \mapsto \beta(x)$ on objects
  and $f \mapsto \alpha(F(f))$ on morphisms.
\end{theorem}
\begin{proof}
  First we verify that the assignment of the statement is
  well-defined, \ie that $\alpha(F(f))$ is indeed a morphism of 
  $P \rtimes G$. Given $f \colon x \to y$, we need to show
  that $\beta(x) \leq \alpha(F(f)) \cdot \beta(y)$. Unfolding the
  definition of action, this means finding an isomorphism $k
  \colon x' \to \beta'(\beta(y))$ with $\alpha(F(k))=\alpha(F(f))$ and
  $\beta(x) \leq \beta(x')$. Unfolding the definition of the preorder,
  the latter means finding a morphism $h' \colon \beta'(\beta(x)) \to
  \beta'(\beta(x'))$ with $F(h')=1$. By (A5), it suffices to find $h
  \colon x \to x'$ with $F(h)=1$ instead. But (A6) provides an
  isomorphism $k \colon x' \to \beta'(\beta(y))$ with
  $\alpha(F(k))=\alpha(F(f))$. By (A5) again, there exists an
  isomorphism $l \colon y \to \beta'(\beta(y))$ with
  $\alpha(F(l))=1$. Finally, we can take $h = k^{-1} l f \colon x \to
  x'$. This morphism indeed satisfies $\alpha(F(h)) = \alpha(F(f)) \cdot \alpha(F(l))
  \cdot \alpha(F(k))^{-1} = \alpha(F(k)) \cdot \alpha(F(k))^{-1} = 1$.

  Functoriality follows directly from the previous lemma, so indeed we
  have a well-defined functor $\cat{A} \to P \rtimes G$. Moreover, our
  functor is essentially surjective because $\beta$ is an equivalence,
  and it is faithful because $F$ is faithful.  

  Finally, to prove fullness, let $m \colon \beta(x) \to \beta(y)$ be a
  morphism in $P \rtimes G$. This means that $\beta(x) \leq
  m \beta(y)$, which unfolds to: there are a morphism $f \colon 
  x \to z$ and an isomorphism $h \colon z \to \beta'(\beta(y))$ in
  $\cat{A}$ with $\alpha(F(f))=1$ and $\alpha(F(h))=m$. By (A5), this is
  equivalent to the existence of a morphism $f \colon x \to z$ with
  $\alpha(F(f))=1$ and an isomorphism $h \colon z \to y$ in $\cat{A}$ with 
  $\alpha(F(h))=m$. Now take $k=hf \colon x \to y$ in $\cat{A}$. Then
  \[
  \alpha(F(k)) = \alpha(F(hf)) = \alpha(F(f)) \cdot \alpha(F(h)) = 1 \cdot m = m.
  \]
  Hence our functor is full, and we conclude that it is (half
  of) an equivalence. 
  \qed
\end{proof}

\section{The finite-dimensional case}\label{sec:finite}

This section uses poset-group-amalgamations to completely
characterize the category $\Cm(A)$ for finite-dimensional C*-algebras
$A$. \emph{En passant}, we will also characterize the poset category $\Cs(C)$ for
commutative finite-dimensional C*-algebras $C$.

\subsection*{Finite partition lattices}
We start with identifying the appropriate poset $P$.
Recall that a \emph{partition} $p$ of $\{1,\ldots,n\}$ is a family of
disjoint subsets $p_1,\ldots,p_k$ of $\{1,\ldots,n\}$ whose union is
$\{1,\ldots,n\}$. Partitions are ordered by \emph{refinement}: $p \leq
q$ whenever each $p_i$ is contained in a $q_j$. Ordered this way, the
partitions of $\{1,\ldots,n\}$ form a lattice, called the
\emph{partition lattice}, that we denote by $P(n)$.
It is known when a lattice is (isomorphic to) the partition
lattice $P(n)$. We recall such a characterization below, but first we
briefly have to recall some terminology.

Recall that a lattice is \emph{semimodular} if $a \vee b$ covers $b$
whenever $a$ covers $a \wedge b$. A finite lattice is \emph{geometric} when
it is atomic and semimodular. 
Any geometric lattice has a well-defined \emph{rank} function: $\rk(x)$
is the length of a(ny) chain from $0$ to $x$ in $L$. 
An element $x$ in a lattice is
\emph{modular} when $a \vee (x \wedge y) = (a \vee x) \wedge y$ for
all $a \leq y$. 
The \emph{M{\"o}bius function} of a finite lattice is the unique function
$\mu \colon L \to \field{Z}$ satisfying $\sum_{y < x} \mu(x) =
\delta_{0,x}$. It can be defined recursively by $\mu(0)=1$ and $\mu(x)
= -\sum_{y \leq x} \mu(x)$ for $x>0$;
see~\cite{blasssagan:mobius}. The \emph{characteristic 
  polynomial} of a finite lattice $L$ is $\sum_{x \in L} \mu(x) \cdot
\lambda^{\rk(1)-\rk(x)}$. Finally, we write ${\uparrow} x$ for the principal ideal $\{z \in L \mid x \leq z \}$ of $x \in L$.

\begin{theorem}\label{yoon}
  A lattice $L$ is isomorphic to $P(n+1)$ if and only if:
  \begin{enumerate}
  \item[(P1)] it is geometric; 
  \item[(P2)] if $\rk(x)=\rk(y)$, then ${\uparrow} x \cong {\uparrow} y$;
  \item[(P3)] it has a modular coatom;
  \item[(P4)] its characteristic polynomial is $(\lambda-1) \cdots (\lambda-n)$.
  \end{enumerate}
\end{theorem}
\begin{proof}
  See~\cite{yoon:partitionlattices}.
  \qed
\end{proof}

This immediately extends to a characterization of $\Cs(A)$ for
finite-dimensional commutative C*-algebras $A$ (which are always unital).

\begin{corollary}
  A lattice $L$ is isomorphic to $\Cs(A)\op$ for a commutative
  C*-algebra $A$ of dimension $n+1$ if and only if it satisfies (P1)--(P4).
\end{corollary}
\begin{proof}
  The lattice $\Cs(A)$ is that of subobjects of $A$ in the category of
  finite-dimensional commutative C*-algebras and unital
  $*$-homomorphisms. Recall that a \emph{subobject} is an equivalence
  class of monomorphisms into a given object, where two monics are identified
  when they factor through one another by an isomorphism. The dual
  notion is a \emph{quotient}: an equivalence class of epimorphisms out of a
  given object. By Gelfand duality, $\Cs(A)$ is isomorphic to the opposite
  of the lattice of quotients of the discrete topological space
  $\Spec(A)$ with $n+1$ points. But the latter is precisely $P(n+1)\op$.
  \qed
\end{proof}

\subsection*{Symmetric group actions}
The appropriate group to consider is the symmetric group $S(n)$ of all
permutations $\pi$ of $\{1,\ldots,n\}$. The group $S(n)$ acts on
$P(n)$. Explicitly,
$\pi p = (\pi p_1,\ldots,\pi p_k)$ for $p=(p_1,\ldots,p_k) \in P(n)$
and $\pi \in S(n)$, where $\pi p_l = \{ \pi(i) \mid i \in \pi_l
\}$. 
That is, one works in the quotient group of $S(n)$ by the Young subgroups $S(n_1) \times \cdots \times S(n_k)$, where the $n_l$ ares the cardinality of the parts $p_l$ of the partition $p$.
The following lemma might be considered the main insight of this
article. 

\begin{lemma}\label{lem:charcommfinite}
  If $A$ is a commutative C*-algebra of dimension $n$, then there is
  an isomorphism $\Cm(A)\op \cong P(n) \rtimes S(n)$ of categories.
\end{lemma}
\begin{proof}
  We may assume that $A=\mathbb{C}^n$. Objects $C$ of $\Cm(A)$ then
  are of the form $C=\{(x_1,\ldots,x_n) \in \mathbb{C}^n \mid \forall
  k \forall i,j \in p_k \colon x_i = x_j \}$ for some partition
  $p=(p_1,\ldots,p_l)$ of $\{1,\ldots,n\}$. But these are precisely
  the objects of $P(n)$, and hence of $P(n) \rtimes S(n)$.

  If $f \colon C' \to C$ is a morphism of $\Cm(A)$, \ie an injective
  $*$-homomorphism, then $f(C') \subseteq C$ is a C*-subalgebra. Say $C'
  = \{ x \in \mathbb{C}^n \mid \forall k \forall i,j \in p'_k \colon
  x_i = x_j \}$ for a partition $p'=(p'_1,\ldots,p'_{l'})$. Then we
  see that $f$ must be induced by an injective function
  $\{1,\ldots,n\} \to \{1,\ldots,n\}$, which we can extend to a
  permutation $\pi \in S(n)$. Then $C' \to C$ means that $\pi p' \leq
  p$. But this is precisely a morphism in $(P(n) \rtimes S(n))\op$.
  \qed
\end{proof}

\subsection*{Terminal subalgebras}

A maximal abelian subalgebra $D$ of a C*-algebra $A$ is a maximal
element in $\Cs(A)$. If $A$ is finite-dimensional, such $D$ are unique
up to conjugation with a unitary.

The prime example is the following:
if $A$ is the C*-algebra $\mathbb{M}_n(\mathbb{C})$ of $n$-by-$n$ complex
matrices, then maximal abelian subalgebras $D$ are precisely the
subalgebras consisting of all matrices that are diagonal in some fixed basis. 

In finite dimension, maximal elements of $\Cs(A)$ are the same as
terminal objects of $\Cm(A)$. For the following
lemma, weakly terminal objects of $\Cs(A)$ are in fact enough. Recall
that an object $D$ is weakly terminal when every object $C$ allows a
morphism $C \to D$.

\begin{lemma}\label{terminal}
  If $\Cm(A)$ has a weak terminal object $D$, then there is an
  equivalence $\Cm(A) \simeq \Cm(D)$ of categories.
\end{lemma}
\begin{proof}
  Clearly the inclusion $\Cm(D) \hookrightarrow \Cm(A)$ is a full and
  faithful functor, so it suffices to prove that it is essentially
  surjective. Let $C \in \Cm(A)$. Then there exists an injective
  $*$-homomorphism $f \colon C \to D$ because $D$ is weakly
  terminal. Hence $C \cong f(C) \in \Cm(D)$.  
  \qed
\end{proof}

\subsection*{The characterization}

We can now bring all the pieces together.

\begin{theorem}\label{charffactor}
  For a category $\cat{A}$, the following are equivalent:
  \begin{itemize}
  \item the category $\cat{A}$ is equivalent to $\Cm(\mathbb{M}_n(\field{C}))\op$;
  \item the category $\cat{A}$ is equivalent to $P(n) \rtimes S(n)$;
  \item 
    $\cat{A}$ satisfies (A1)--(A6), and \\
    \hspace*{2mm}$(\cat{A}, \leq)$ satisfies (P1)--(P4) for $n-1$, and \\
    \hspace*{2mm}$\cat{A}(0,0)\op$ is isomorphic to the symmetric group on $n$ elements. 
  \end{itemize}
\end{theorem}
\begin{proof}
  Combine the previous two lemmas with Theorem~\ref{charPG} and
  Theorem~\ref{yoon}.
  \qed
\end{proof}

We can actually do better than characterizing
factors $A=\mathbb{M}_n(\field{C})$ of type $\mathrm{I}_n$: the next theorem characterizes $\Cm(A)$ for any
finite-dimensional C*-algebra $A$.

\begin{lemma}\label{directsums}
  If $\Cm(A_i)$ has a weak terminal object $D_i$ for each $i$ in a set
  $I$, then the C*-direct sum $\bigoplus_{i \in I} D_i$ is a weak terminal object in
  $\Cm(\bigoplus_{i \in I} A_i)$.   
\end{lemma}
\begin{proof}
  Let $C \in \C(\bigoplus_{i \in I} A_i)$. 
  Then $C$ is contained in the commutative subalgebra $\bigoplus_{i
    \in I} \pi_i(C)$ of $\bigoplus_{i \in I} A_i$. Because each $D_i$
  is weakly terminal, there exist morphisms $f_i \colon \pi_i(C) \to
  D_i$. Therefore $\bigoplus_{i \in I} f_i $ is a morphism
  $\bigoplus_{i \in I} \pi_i(C) \to \bigoplus_{i \in I} D_i$, 
  and thus the latter is weakly terminal in $\Cm(\bigoplus_{i \in I} A_i)$.
  \qed
\end{proof}

\begin{theorem}\label{charfinite}
  A category $\cat{A}$ is equivalent to $\Cm(A)\op$ for a
  finite-dimensional C*-algebra $A$ if and only if there are
  $n_1,\ldots,n_k \in \field{N}$ such that:
  \begin{itemize}
  \item $\cat{A}$ satisfies (A1)--(A5) and (A6');
  \item $(\cat{A},\leq)$ satisfies (P1)--(P4) for $(\sum_{i=1}^k n_i)-1$;
  \item $\cat{A}(0,0)\op$ is isomorphic to the symmetric group on
    $\sum_{i=1}^k n_i$ elements;
  \item $\sum_{i=1}^k n_i^2 = \dim(A)$.
  \end{itemize}
\end{theorem}
\begin{proof}
  Every finite-dimensional C*-algebra $A$ is isomorphic to a matrix realization of the form $\bigoplus_{i=1}^k \mathbb{M}_{n_i}(\field{C})$ with $n=\sum_{i=1}^k
  n_i^2$~\cite[Theorem~III.1.1]{davidson:cstar}. By
  Lemmas~\ref{terminal} and~\ref{directsums}, we have 
  \[
  \Cm(A) 
  \simeq \Cm(\bigoplus_{i=1}^k \field{C}^{n_i}) 
  \cong \Cm(\field{C}^{(\sum_{i=1}^k n_i)}).
  \]
  So by Lemma~\ref{lem:charcommfinite}, $\Cm(A)\op \simeq P(m) \rtimes S(m)$ for
  $m=\sum_{i=1}^k n_i$. Now the
  statement follows from Theorem~\ref{charffactor}.
  \qed
\end{proof}

Notice that by Lemma~\ref{directsums} we may indeed use the whole partition lattice $P(m)$ in the previous theorem instead of the truncated one $P(n_1) \times \cdots \times P(n_k)$; this is one of the consequences of working with equivalences of categories instead of isomorphisms.

\section{Poset-monoid-amalgamations}\label{sec:monoidamalgamations}

The main idea of our characterization of $\Cm(A)$ for
finite-dimensional C*-algebras $A$ holds unabated in the
infinite-dimensional case. However, the technical implementation of
the idea needs some adapting. For example, the appropriate monoid is
no longer a group. Therefore, we will have to refine axiom (A6) into
(A6') and (A7') as follows. We re-list the other axioms for convenience.

\begin{definition}
  A category $\cat{A}$ is called a \emph{poset-monoid-amalgamation}
  when there exist a partial order $P$ and a monoid $M$ such that:
  \begin{enumerate}
  \item[(A1')] there is a weak initial object $0$, unique up to isomorphism;
  \item[(A2')] there is a faithful retraction $F$ of the inclusion
    $\cat{A}(0,0) \hookrightarrow \cat{A}$;
  \item[(A3')] there is an isomorphism $\alpha \colon \cat{A}(0,0) \to M\op$ of
    monoids;
  \item[(A4')] there is an equivalence $\smash{\xymatrix@1{(\cat{A},\leq)
        \ar@<1ex>|-{\beta}[r] & P \ar@<.5ex>|-{\beta'}[l]}}$ of preorders;
  \item[(A5')] for each object $x$ there is an isomorphism $f \colon x
    \to \beta'(\beta(x))$ with $\alpha F(f)=1$;
  \item[(A6')] for each object $y$ and $m \colon 0 \to 0$,
    there is $f \colon x \to y$ such that $F(f)=m$, that is universal
    in the sense that $f'=fg$ with
    $F(g)=1$ for any $f' \colon x' \to y$ with $F(f')=m$; 
  \item[(A7')] if $F(f)=m_2m_1$ for a morphism $f$, then $f=f_1f_2$ with $F(f_i)=m_i$.
  \end{enumerate}
\end{definition}

The idea behind axiom (A6') is that in $P \rtimes M$, we can access
the object $mq$ through the morphisms $m \colon p \to q$. There is
always at least one such morphism, namely $m \colon mq \to q$, because
trivially $mq \leq mq$. This might not be an isomorphism, but it still
has the universal property that all other morphisms $m \colon p \to
q$ factor through it.
We can rephrase this universality as follows: for each object $y$ of
$P \rtimes M$ and $m \in M$, there is a maximal element of the set $\{ f
\colon x \to y \mid \alpha(F(f))=m\}$, preordered by $f \leq g$ iff $f=hg$ for
some morphism $h$ satisfying $\alpha(F(h))=1$.
\[\xymatrix{
  x \ar^-{f}[r] & y \\
  z \ar_-{g}[ur] \ar@{-->}^-{h}[u]
}\]
Also, notice the swap in (A7'). It is caused by the contravariance in the
composition of $P \rtimes M$ and (A3'), and is not a mistake, as the
following example shows.

\begin{example}
  If $P$ is a partial order with least element, and $M$ is a
  monoid acting on $P$, then $P \rtimes M$ is a poset-monoid-amalgamation.
\end{example}
\begin{proof}
  The least element $0$ of $P$ is a weak initial object, satisfying
  (A1'). Conditions (A2')--(A4') are satisfied by definition, and (A5') is vacuous.  
  To verify (A6') for $q \in P$ and $m \in M$, notice that $mq \leq
  mq$, and if $p \leq mq$, then certainly $p \leq 1mq$. Finally, (A7')
  means that if $p \leq m_2 m_1 r$, we should be able to provide $q$
  such that $p \leq m_2 q$ and $q \leq m_1 r$; taking $q=m_1 r$
  suffices. 
  \qed
\end{proof}

\begin{lemma}\label{recoverMaction}
  If $\cat{A}$ satisfies (A1')--(A7'), then it induces an 
  action of $M$ on $P$ given by $pm = \beta(x)$ if $f \colon x \to
  \beta'(p)$ is a maximal element with $\alpha(F(f))=m$.
\end{lemma}
\begin{proof}
  First, notice that for any $p \in P$ and $m \in M$ there exists a
  maximal $f \colon x \to \beta'(p)$ with $\alpha(F(f))=m$ by
  (A6'). If there is another maximal $f' \colon x' \to \beta'(p)$ with
  $\alpha(F(f'))=m$, then there are morphisms $g \colon x \to x'$ and
  $g' \colon x' \to x$ with $F(g)=1=F(g')$. Hence $F(gg')=1=F(g'g)$,
  and because $F$ is faithful, $g$ is an isomorphism with $g'$ as
  inverse. So $x \cong x'$, and therefore
  $\beta(x)\cong\beta(x')$. But because $P$ is a partial order, this
  means $\beta(x)=\beta(x')$. Thus the action is well-defined on objects.

  To see that it is well-defined on morphisms, suppose that $p \leq
  q$. Then there is a morphism $f \colon \beta'(p) \to \beta'(q)$ with
  $F(f)=1$. For any $m \colon 0 \to 0$, we can find maximal $f_p
  \colon x_p \to \beta'(p)$ with $F(f_p)=m$, and maximal $f_q \colon
  x_q \to \beta'(q)$ with $F(f_q)=m$. Now $ff_p \colon x_p \to
  \beta'(q)$ has $F(ff_p)=m$. Because $f_q$ is a maximal such
  morphism, $ff_p$ factors through $f_q$. That is, there is $h \colon
  x_p \to x_q$ with $f_qh=ff_p$ and $F(h)=1$. So $mp \leq mq$.

  Next, we verify that this assignment is functorial $M \to
  \Cat{Cat}(P,P)$. Clearly $\id[\beta'(p)]$ is maximal among morphisms
  $f \colon x \to \beta'(p)$ with $F(f)=1$. Therefore $1p=\beta(\beta'(p))=p$.

  For $m_2,m_1 \in M$ and $p \in P$, we have $m_1p=\beta(x_1)$ where $f_1
  \colon x_1 \to \beta'(p)$ is maximal such that $\alpha(F(f_1))=m_1$. Therefore
  $m_2(m_1p) = \beta(x_2)$, where the morphism $f_2 \colon x_2 \to
  \beta'(\beta(x_1))$ is maximal such that $\alpha(F(f_2))=m_2$. By axiom (A5'),
  there is an isomorphism $h \colon x_1 \to \beta'(\beta(x_1))$ with
  $F(h)=1$. This gives $f =f_1h^{-1}f_2 \colon x_2
  \to \beta'(p)$ with $\alpha(F(f))=\alpha(F(f_2)) \cdot
  \alpha(F(h))^{-1} \cdot \alpha(F(f_1)) = m_2m_1$. We will now show
  that $f$ is universal with this property. If $g \colon y \to
  \beta'(p)$ also has $\alpha(F(g))=m_2m_1$, then (A7') provides $g_2
  \colon y \to z$ and $g_1 \colon z \to \beta'(p)$ with $g=g_1g_2$ and
  $\alpha(F(g_i))=m_i$. 
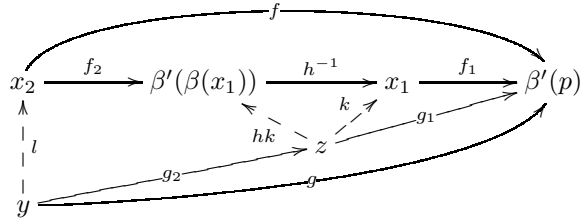
  \[\xymatrix@C-2ex@R-3ex{
    x_2 \ar^-{f_2}[rr] \ar@(u,u)|-{f}[rrrrrr]
    && \beta'(\beta(x_1)) \ar^-{h^{-1}}[rr] 
    && x_1 \ar^-{f_1}[rr]
    && \beta'(p) \\
    &&& z \ar@{-->}^-{hk}[ul] \ar@{-->}^-{k}[ur] \ar|-{g_1}[urrr] \\
    y \ar@{-->}_-{l}[uu] \ar|-{g_2}[urrr] \ar@(r,d)|-{g}[uurrrrrr]
  }\]
  By maximality of $f_1$, there exists $k$ with $g_1=f_1k$ and
  $\alpha(F(k))=1$. And by maximality of $f_2$, there is exists $l$ with
  $hkg_2=f_2l$ and $\alpha(F(l))=1$. Hence
  \[
  g=g_1g_2=f_1kg_2=f_1h^{-1}hkg_2=f_1h^{-1}f_2l=fl. 
  \]
  So $f$ is
  maximal with $F(f)=m_2m_1$. Thus $(m_2m_1)p=\beta(x_2)=m_2(m_1p)$.
  \qed
\end{proof}

\begin{theorem}\label{charPM}
  If $\cat{A}$ satisfies (A1')--(A7'), then there is an equivalence
  $\cat{A} \to P \rtimes M$ given by $x \mapsto \beta(x)$ on objects
  and $f \mapsto \alpha(F(f))$ on morphisms.
\end{theorem}
\begin{proof}
  First, it follows from (A6') that the assignment of the statement is
  well-defined, \ie that $\alpha(F(f))$ is indeed a morphism of 
  $P \rtimes M$. Indeed, if $f \colon x \to y$, then we need to show
  that $\beta(x) \leq \alpha(F(f)) \cdot \beta(y)$. Unfolding the
  definition of the action, this means we need to find a maximal $k
  \colon x' \to \beta'(\beta(y))$ with $F(f)=F(k)$, such that
  $\beta(x) \leq \beta(x')$. Unfolding the definition of the preorder,
  this means we need to find a morphism $h' \colon \beta'(\beta(x)) \to
  \beta'(\beta(x'))$ with $F(h')=1$. By (A5'), it suffices to find $h
  \colon x \to x'$ with $F(h)=1$ instead. But by (A6'), there exists
  maximal $k \colon x' \to \beta'(\beta(y))$ with $F(k)=F(f)$. By its
  maximality, there exists $h \colon x \to x'$ with $F(h)=1$ and
  $f=kh$. In particular, $\beta(x) \leq \beta(x')$.

  Functoriality follows directly from the previous lemma, so indeed we
  have a well-defined functor $\cat{A} \to P \rtimes M$. Moreover, our
  functor is essentially surjective because $\beta$ is an equivalence,
  and it is faithful because $F$ is faithful.  

  Finally, to prove fullness, let $m \colon \beta(x) \to \beta(y)$ be a
  morphism in $P \rtimes M$. This means that $\beta(x) \leq
  \beta(y) m$, which unfolds to: there are morphisms $f \colon 
  x \to z$ and $h \colon z \to \beta'(\beta(y))$ with $F(f)=1$ and $h$
  maximal with $\alpha(F(h))=m$. By (A5'), this is 
  equivalent to the existence of a morphism $f \colon x \to z$ with
  $F(f)=1$ and a morphism $h \colon z \to y$ maximal with 
  $\alpha(F(h))=m$. Now take $k=hf \colon x \to y$. Then
  \[
  \alpha(F(k)) = \alpha(F(hf)) = \alpha(F(f))\alpha(F(h)) = 1 \cdot m
  = m.
  \]
  Hence our functor is full, and we conclude that it is (half
  of) an equivalence. 
  \qed
\end{proof}

\section{The infinite-dimensional case}\label{sec:infinite}

To adapt Theorem~\ref{charffactor} to the infinite-dimensional case,
we have to make three more adaptations. First, the poset $P$ now becomes
a lattice of partitions of a (locally) compact Hausdorff
space. Second, the symmetric group gets replaced by symmetric monoids
on (locally) compact Hausdorff spaces. Third, we have to be more
careful about maximal abelian subalgebras. 

\subsection*{Infinite partition lattices}

For arbitrary (locally) compact Hausdorff spaces, it is more
convenient to talk about equivalence relations than about partitions.
An equivalence relation $\sim$ on a (locally) compact Hausdorff space
$X$ is called \emph{closed}  when the set $\{x \in X \mid \exists u
\in U .\, x \sim u \}$ is closed for every closed $U \subseteq
X$. Closed equivalence relations on $X$ are also called
\emph{partitions}, and form a partial order $P(X)$ under \emph{refinement}:
\[
\mathop{\sim} \leq \mathop{\approx} 
\quad \iff \quad
\big(\forall x,y \in X .\, x \sim y \implies x \approx y \big).
\]
Notice that quotients of a (locally) compact Hausdorff space by an
equivalence relation are again (locally) compact Hausdorff if and only if the equivalence relation is closed. 

Fortunately, a characterization of $P(X)$ is known, due to
Firby~\cite{firby:compactifications1,firby:compactifications2}. 
This also gives a characterization of $\Cs(A)$ for 
commutative C*-algebras $A$. As in Section~\ref{sec:finite}, we first
briefly recall the necessary terminology.
An element $b$ of a lattice is called \emph{bounding} when (i) it is zero
or an atom; or (ii) it covers an atom and dominates exactly three
atoms; or (iii) for distinct atoms $p,q$ there exists an atom $r \leq
b$ such that there are exactly three atoms less than $r \vee p$ and
exactly three atoms less than $r \vee q$. A collection of atoms of a
lattice with at least four elements is called \emph{single} when it is a
maximal collection of atoms of which the join of any two dominates
exactly three atoms (not necessarily in the collection).
A collection $B$ of nonzero bounding elements of a lattice is called a
\emph{1-point} when (i) its atoms form a single collection; and (ii) if
$a$ is bounding and $a \geq b \in B$, then $a \in B$; and (iii) any $a
\in B$ dominates an atom $p \in B$.

\begin{theorem}\label{charPX}
  A lattice $L$ with at least four elements is isomorphic to $P(X)$
  for a compact Hausdorff space $X$ if and only if:
  \begin{enumerate}
  \item[(P1')] $L$ is complete and atomic;
  \item[(P2')] the intersection of any two 1-points contains exactly one atom, \\ 
    and any atom belongs to exactly two 1-points;
  \item[(P3')] for bounding $a,b \in L$ that are contained in a 1-point,
    \begin{align*}
      \quad\qquad & \{ p \in \mathrm{Atoms}(L) \mid p \leq a \vee b \} \\ & = 
      \{ p \in \mathrm{Atoms}(L) \mid \text{ if } x \text{ is a 1-point with } p
      \in x \text{ then } a \in x \text{ or } b \in x \};
    \end{align*}
    for bounding $a,b \in L$ that are not contained in a 1-point,
    \[
    \;
    \{ p \in \mathrm{Atoms}(L) \mid p \leq a \vee b \} = 
    \{ p \in \mathrm{Atoms}(L) \mid p \leq a \text{ or } p \leq b \};
    \]
  \item[(P4')] for 1-points $x\neq y$ there are bounding $a,b$ with $a
    \not\in x$, $b \not\in y$, and $a \vee b=1$;
  \item[(P5')] joins of bounding elements are bounding;
  \item[(P6')] for nonzero $a \in L$, the collection $B$ of bounding
    elements equal to or covered by $a$ is the unique one satisfying:
    \begin{itemize}
    \item $\bigvee B = a$;
    \item no 1-point contains two members of $B$;
    \item if $c$ is bounding, $b_1 \in B$, and no 1-point contains
      $b_1$ and $c$, then there is a bounding $b \geq c$ such that (i)
      there is no 1-point containing both $b$ and $b_1$, and (ii) 
      whenever there is a 1-point containing both $b$ and $b_2 \in B$,
      then $b \geq b_2$;
    \end{itemize}
  \item[(P7')] any collection of nonzero bounding elements that is not
      contained in a 1-point has a finite subcollection that is not contained in a 1-point;
  \end{enumerate}
  and $X$ is (homeomorphic to) the set of 1-points of
  $L$, where a subset is closed if it is a singleton 1-point or it
  is the set of 1-points containing a fixed bounding element.
\end{theorem}
\begin{proof}
  See~\cite{firby:compactifications2}.
  \qed
\end{proof}

\begin{remark}
  The axiom responsible for compactness of $X$ is (P7'). The previous
  theorem holds for locally compact Hausdorff spaces $X$ when we replace (P7') by
  \begin{itemize}
  \item[(P7'')] every 1-point contains a bounding $b$ such that
    $\{ l \in L \mid l \geq b \}$ satisfies (P7').
  \end{itemize}
  Indeed, because (P1')--(P6') already guarantee Hausdorffness, we may
  take local compactness to mean that every point has a compact
  neighbourhood that is closed. And closed sets correspond to sets of
  1-points containing a fixed bounding element.
\end{remark}

As before, this directly leads to a characterization of $\Cs(A)$ for
commutative C*-algebras $A$.

\begin{corollary}
  A lattice $L$ is isomorphic to $\Cs(A)\op$ for a commutative
  C*-algebra $A$ of dimension at least three if and only if it
  satisfies (P1')--(P6') and (P7''). The C*-algebra $A$ is unital if and
  only if $L$ additionally satisfies (P7').
\end{corollary}
\begin{proof}
  The lattice $\Cs(A)$ is that of subobjects of $A$ in the category of commutative
  (unital) C*-algebras and (unital) nondegenarate
  $*$-homomorphisms. Recall that a \emph{subobject} is an equivalence class of
  monomorphisms into a given object, where two monics are identified
  when they factor through one another by an isomorphism. The dual
  notion is a \emph{quotient}: an equivalence class of epimorphisms out of a
  given object. By Gelfand duality, $\Cs(A)$ is isomorphic to the opposite
  of the lattice of quotients of $X=\Spec(A)$. 
  But the latter is precisely $P(X)\op$, because categorical
  quotients in the category of (locally) compact Hausdorff spaces are
  quotient spaces. 
  \qed
\end{proof}


\subsection*{Symmetric monoid actions}

We write $S(X)$ for the monoid of continuous functions 
$f \colon X \twoheadrightarrow X$ with dense image on a locally
compact Hausdorff space $X$, called the \emph{symmetric monoid} on $X$.

The monoid $S(X)$ acts on $P(X)$, as described in the following lemma. We stick to the notation $mx$ for the action of an element $m$ of a monoid $M$ on an object $x$ of a category as in Definition~\ref{action}. For $f \in S(X)$ and $\mathop{\sim} \in P(X)$ this becomes $f \mathop{\sim}$.

\begin{proposition}\label{SXaction}
  For any locally compact Hausdorff space $X$, the monoid $S(X)$ acts
  on $P(X)$ by 
  \[
  (f \mathop{\sim}) \;=\; (f \times f)^{-1}(\mathop{\sim}).
  \]
\end{proposition}
\begin{proof}
  First of all, notice that $f\mathop{\sim}$ is reflexive, symmetric and 
  transitive, so indeed is a well-defined equivalence relation on $X$, which is closed
  because $f$ is continuous. Concretely, $x (f\mathop{\sim}) y$ if and only if $f(x) \sim
  f(y)$. Moreover, clearly $\id\mathop{\sim} = \mathop{\sim}$, and $g(f\mathop{\sim}) =
  (gf)\mathop{\sim}$, so the above is a genuine action. 
  \qed
\end{proof}

As before, this directly leads to a characterization of $\Cm(A)$ for
commutative C*-algebras $A$.

\begin{lemma}\label{charCmcomm}
  If $A=C(X)$ for a locally compact Hausdorff space $X$, there is an isomorphism
  $\Cm(A)\op \cong P(X) \rtimes S(X)$ of categories.
\end{lemma}
\begin{proof}
  By definition, objects $C$ of $\Cm(A)$ are subobjects of $C(X)$ in the category of
  commutative C*-algebras. By Gelfand duality, these correspond to quotients of $X$ in the
  category of locally compact Hausdorff spaces. But these, in turn, correspond to closed
  equivalence relations on $X$, establishing a bijection between the objects of $\Cm(A)$
  and $P(X)$. 

  Through Gelfand duality, a morphism $C \rightarrowtail C'$ in $\Cm(A)$ correspondsto an
  epimorphism $g \colon Y' \twoheadrightarrow Y$ between the corresponding spectra.
  Writing the quotients as $Y=X/\mathop{\sim}$ and $Y'=X/\mathop{\approx}$ for closed
  equivalence relations $\sim$ and $\approx$, we find that $g$
  corresponds to a continuous $f \colon X \twoheadrightarrow X$ with dense
  image respecting equivalence: 
  \[
  x \approx y \implies f(x) \sim f(y).
  \]
  But this just means $\mathop{\approx} \leq (f \mathop{\sim})$.
  In other words, this is precisely a morphism $f \colon \mathop{\approx} \to 
  \mathop{\sim}$ in $P(X) \rtimes S(X)$.
  \qed
\end{proof}

\subsection*{Weakly terminal subalgebras}

In the infinite-dimensional case, it is no longer true that all
maximal abelian subalgebras of a C*-algebra $A$ are
isomorphic. However, it suffices if there exists a commutative
subalgebra into which all others embed by an injective
$*$-homomorphism. To be
precise, a commutative C*-subalgebra $D$ of a C*-algebra $A$ is
\emph{weakly terminal} when each $C \in \C(A)$ allows an injective
$*$-homomorphisms $C \to D$ (that is not necessarily an inclusion, and
not necessarily unique). Equivalently, every masa is isomorphic to a
subalgebra of $D$. Weakly terminal commutative subalgebras $D$ are
maximal up to isomorphism, in the sense that if $D$ can be extended to
a larger commutative C*-subalgebra $E$, then $D \cong E$. This does
not imply that $D=E$, \ie that $D$ is maximal. For a counterexample,
take $A=E=C([0,1])$ and $D=\{f \in E \mid f \mbox{ constant on
}[0,\tfrac{1}{2}]\}$. Then $D \subsetneq E$, but $D \cong
C([\tfrac{1}{2},1]) \cong E$.

\begin{lemma}\label{separablemasas}
  If $A=B(H)$ for an infinite-dimensional separable Hilbert space $H$,
  then $\Cm(A)$ has a weak terminal object, $*$-isomorphic to
  $L^\infty(0,1) \oplus \ell^\infty(\field{N})$.
\end{lemma}
\begin{proof}
  Let $C \in \Cm(A)$. By Zorn's lemma, $C$ is a C*-subalgebra of a
  maximal element of $\Cs(A)$.
  A maximal element in $\Cs(A)$ for a von Neumann algebra $A$ is itself
  a von Neumann algebra, because it must equal its weak
  closure. It is known that maximal abelian von Neumann subalgebras of
  $A=B(H)$ for an infinite-dimensional separable Hilbert space $H$ are
  unitarily equivalent to one of the following: $L^\infty(0,1)$,
  $\ell^\infty(\{0,\ldots,n\})$ for $n \in \field{N}$, 
  $\ell^\infty(\field{N})$, $L^\infty(0,1) \oplus
  \ell^\infty(\{0,\ldots,n\})$ for $n \in \field{N}$, or $L^\infty(0,1)
  \oplus \ell^\infty(\field{N})$
  (see~\cite[Theorem~9.4.1]{kadisonringrose:operatoralgebras}). 
  Each of these allows an injective $*$-homomorphism into the latter
  one, $D=L^\infty(0,1) \oplus \ell^\infty(\field{N})$.
  Therefore, there exists a morphism $C \to D$ in $\Cm(A)$ for each
  $C$ in $\Cm(A)$, so that $D$ is weakly terminal in $\Cm(A)$.
  \qed
\end{proof}

If $\dim(H)$ is uncountable, the situation becomes a bit more involved. A
complete classification of (maximal) abelian subalgebras of $B(H)$ is
known~\cite{segal:decomposition,sinclairsmith:masas}, and we will use this to establish a
weakly terminal commutative subalgebra in the following lemma.
Before doing so, let us explain the intuitition behind the use of cardinal numbers $\alpha$ and $\beta$ in the statement. For any cardinal number $\alpha$, the C*-algebra $B(H)$ has a commutative subalgebra $L^\infty(0,1)^\alpha$ that needs to be accounted for in a weakly terminal commutative subalgebra, as in the previous lemma. Because there are $\dim(H)$ many of those,  a sum over a second cardinal $\beta\leq \dim(H)$ is called for.

\begin{lemma}\label{masas}
  If $A=B(H)$ for an infinite-dimensional Hilbert space $H$,
  then $\Cm(A)$ has a weak terminal object, $*$-isomorphic to
  $\bigoplus_{\alpha,\beta \leq \dim(H)} L^\infty\big( (0,1)^\alpha
  \big)$, where $\alpha,\beta$ are cardinals, and $(0,1)^\alpha$ is
  the product measure space of Lebesgue unit intervals.
\end{lemma}
\begin{proof}
  Maximal abelian subalgebras $C$ of $B(H)$ are isomorphic to direct
  sums of $L^\infty\big( (0,1)^\alpha\big)$ ranging over cardinal
  numbers $\alpha$~\cite{segal:decomposition}. 
  We must show that $D = \bigoplus_{\alpha,\beta \leq
  \dim(H)} L^\infty\big( (0,1)^\alpha \big)$ can be identified with a subalgebra of $B(H)$
  that embeds any such $C$. 
  A commutative algebra $L^\infty\big( (0,1)^\alpha \big)$ acts on the
  Hilbert space $L^2\big( (0,1)^\alpha \big)$. 
  Observe that $L^2(0,1)$ is separable. 
  Hence $\dim\big( L^2\big( (0,1)^\alpha \big)\big) = \max(\alpha,
  \aleph_0)$ unless $\alpha=0$, in which case the dimension
  vanishes. Therefore $\dim \big( L^2\big( (0,1)^\alpha \big) 
  \big) \leq \dim(H)$ if and only if $\alpha \leq \dim(H)$.
  Because $H$ is infinite-dimensional, we have the equation $\dim(H) =
  \dim(H)^3$ of cardinal numbers.  
  Thus any maximal abelian subalgebra $C$ embeds into $D$, and $D$
  itself embeds as a maximal abelian subalgebra of $B(H)$.
  \qed
\end{proof}

The following infinite-dimensional analogue of Lemma~\ref{directsums}
will allow us to deduce from the previous lemma that arbitrary type I
von Neumann algebras possess weakly terminal commutative
subalgebras. (For direct integrals,
see~\cite[Chapter~14]{kadisonringrose:operatoralgebras}.) 

\begin{lemma}\label{directintegrals}
  Let $(M,\mu)$ be a measure space, and for each $t \in M$ let $A_t$
  be a von Neumann algebra. If $\Cm(A_t)$ has a weakly terminal
  object $D_t$ for almost every $t$, then the direct integral
  $\int_M^\oplus D_t \mathrm{d}\mu(t)$ is a weak terminal object in
  $\Cm(\int_M^\oplus A_t \mathrm{d} \mu(t))$.
\end{lemma}
\begin{proof}
  Let $C \in \Cm(\int_M^\oplus A_t \mathrm{d}\mu(t))$.
  Supposing $A_t$ acts on a Hilbert space $H_t$, then $C$ is contained
  in $\int_M^\oplus C_t \mathrm{d}\mu(t)$, where $C_t$ is the von
  Neumann subalgebra of $B(H_t)$ generated by $\{a_t \mid a \in C\}$.
  But because almost every $D_t$ is weakly terminal, this in turn
  embeds into $\int_M^\oplus D_t \mathrm{d}\mu(t)$, which is therefore
  weakly terminal.
  \qed
\end{proof}

\begin{corollary}\label{typeiterminal}
  Every type I von Neumann algebra $A$ possesses a weakly terminal
  commutative subalgebra $D$. More precisely: if $A=\int_M^\oplus A_t
  \mathrm{d}\mu(t)$ for a measure space $(M,\mu)$ and type I factors
  $A_t$ acting on Hilbert spaces $H_t$, then $D$ is $*$-isomorphic to
  $\Spec\big(\int_M^\oplus \bigoplus_{\alpha,\beta\leq \dim(H_t)} L^\infty((0,1)^\alpha) \mathrm{d}\mu(t)\big)$. 
\end{corollary}
\begin{proof}
  Every type I von Neumann algebra is a direct integral of type I
  factors~\cite[Section~14.2]{kadisonringrose:operatoralgebras}. Since
  the latter have weakly terminal commutative 
  subalgebras by Lemma~\ref{masas}, we can deduce that the original
  algebra has a weakly terminal commutative subalgebra by
  Lemma~\ref{directintegrals}. 
  \qed
\end{proof}

Much less is known about the structure of (maximal) abelian
subalgebras of von Neumann algebras of type II and III; see~\cite{bures:masas,sinclairsmith:masas}.
The results of~\cite{tomiyama:masas} indicate that the previous lemma
might extend to show that $\Cm(A)$ has a weak terminal object for
\emph{any} von Neumann algebra $A$. 
It would also be interesting to see if the previous corollary implies
that type I C*-algebras have weakly terminal commutative subalgebras.

\subsection*{The characterization}

We now arrive at our main
result: the characterization $\Cm$ for infinite-dimensional
type I von Neumann algebras.

\begin{theorem}\label{chartypei}
  For a category $\cat{A}$ and an infinite-dimensional type I von Neumann algebra
  $A=\int_M^\oplus B(H_t)\mathrm{d}\mu(t)$ for a measure space $(M,\mu)$
  and Hilbert spaces $H_t$, the following are equivalent:
  \begin{itemize}
  \item the category $\cat{A}$ is equivalent to $\Cm(A)\op$;
  \item the category $\cat{A}$ is equivalent to $P(X) \rtimes S(X)$, \\ \hspace*{1mm}
    where $X$ is the topological space $\Spec\big(\int_M^\oplus
    \bigoplus_{\alpha,\beta\leq \dim(H_t)}
    L^\infty((0,1)^\alpha)\big)$; 
  \item $\cat{A}$ satisfies (A1')--(A7'), and \\
    \hspace*{1mm} $(\cat{A}, \leq)$ satisfies (P1')--(P6'),(P7''), giving a topological
      space $X$, and \\
    \hspace*{1mm} $\cat{A}(0,0)\op$ is isomorphic to the monoid $S(X)$, and \\
    \hspace*{1mm} $X$ is homeomorphic to $\Spec\big(\int_M^\oplus
    \bigoplus_{\alpha,\beta\leq \dim(H_t)}
    L^\infty((0,1)^\alpha) \mathrm{d}\mu(t)\big)$. 
  \end{itemize}
  When $A=B(H)$ for an infinite-dimensional Hilbert space $H$, the
  space $X$ simplifies to $\bigsqcup_{\alpha,\beta\leq \dim(H)}
  \Spec\big(L^\infty\big((0,1)^\alpha\big)\big)$. When $H$ is
  separable, $X$ further simplifies to $\Spec(L^\infty(0,1)) \sqcup
  \Spec(\ell^\infty(\field{N}))$. 
\end{theorem}
\begin{proof}
  Combine the previous four lemmas with Theorems~\ref{charPM}
  and~\ref{charPX}. For the last condition, remember that
  Gelfand duality turns direct sums of commutative C*-algebras into coproducts of
  Hausdorff spaces. 
  \qed
\end{proof}

The Gelfand spectrum of $\ell^\infty(\field{N})$ is the
Stone-{\v C}ech compactification of the discrete topology of
$\field{N}$. In other words, $\Spec(\ell^\infty(\field{N}))$ 
consists of the ultrafilters on $\field{N}$. 
A topological space is homeomorphic to $\Spec(L^\infty(0,1))$ if and
only if it is compact, Hausdorff, totally disconnected, and its clopen
subsets are isomorphic to the Boolean algebra of (Borel) measurable subsets of
the interval $(0,1)$ modulo (Lebesgue) negligible ones.
Since both spaces are compact, we could have used (P7')
instead of (P7'') in the previous theorem for the case $A=B(H)$ with
$H$ separable.

\section{Inclusions versus injections}\label{sec:inclusionsinjections}

This section compares $\Cs$ to $\Cm$. We will show for
C*-algebras $A$ and $B$ that:
\begin{itemize}
\item if $\Cm(A)$ and $\Cm(B)$ are isomorphic, $\Cs(A)$ and
  $\Cs(B)$ are isomorphic;
\item if $\Cm(A)$ and $\Cm(B)$ are equivalent, $\Cs(A)$ and
  $\Cs(B)$ are Morita-equivalent.
\end{itemize}
Here we call two categories $\cat{C}$ and $\cat{D}$ Morita-equivalent when they have equivalent presheaf categories $\PSh(\cat{C}) \simeq \PSh(\cat{D})$.

For any category $\cat{C}$, recall that the category $\int_{\cat{C}}
P$ of elements of a presheaf $P \in \PSh(\cat{C})$ is defined as
follows. Objects are pairs $(C,x)$ of $C \in \cat{C}$ and $x \in
P(C)$. A morphism $(C,x) \to (D,y)$ is a morphism $f \colon C \to D$
in $\cat{C}$ satisfying $x=P(f)(y)$. 
Recall that, for any presheaf $P \in \PSh(\cat{C})$, objects of the slice category $\PSh(\cat{C}) / P$ are natural transformations $\alpha \colon Q \Rightarrow P$ from some presheaf $Q \in \PSh(\cat{C})$ to $P$.

\begin{lemma}
  For any $P \in \PSh(\cat{C})$, the toposes $\PSh(\cat{C})/P$
  and $\PSh(\int_{\cat{C}} P)$ are equivalent.
\end{lemma}
\begin{proof}
  See~\cite[Exercise~III.8(a)]{maclanemoerdijk:sheaves}; we write out
  a proof for the sake of explicitness. Define a functor $F \colon
  \PSh(\cat{C})/P \to \PSh(\int_{\cat{C}} P)$ by 
  \begin{align*}
    F\big(Q \stackrel{\alpha}{\Rightarrow} P\big)(C,x) & = \alpha_C^{-1}(x), \\
    F\big(Q \stackrel{\alpha}{\Rightarrow} P\big)\big( (C,x) \stackrel{f}{\to}
    (D,y) \big) & = Q(f), \\
    F(Q \stackrel{\beta}{\Rightarrow} Q')_{(C,x)} & = \beta_C.
  \end{align*}
  Define a functor $G \colon \PSh(\int_{\cat{C}} P) \to \PSh(\cat{C})/P$ 
  by $G(R)=(Q \stackrel{\alpha}{\Rightarrow} P)$ where
  \begin{align*}
    Q(C) & = \coprod_{x \in P(C)} R(C,x), \\
    Q(C \stackrel{f}{\to} D) & = R\big( (C,P(f)(y)) \stackrel{f}{\to} (D,y) \big), \\
    \alpha_C(\kappa_x(r)) & = x,
  \end{align*}
  where $\kappa_x \colon R(C,x) \to \coprod_{x \in P(C)} R(C,x)$ is
  the coproduct injection. The functor $G$ acts on morphisms as
  \[
  G(R \stackrel{\beta}{\Rightarrow} R')_C = \coprod_{x \in P(C)} \beta_{(C,x)}.
  \]
  Then one finds that $GF(Q \stackrel{\alpha}{\Rightarrow} P) = (Q
  \stackrel{\alpha}{\Rightarrow} P)$, and $FG(R) = \hat{R}$, where 
  \begin{align*}
    \hat{R}(C,x) & = \{ x \} \times R(C,x), \\
    \hat{R}\big( (C,x) \stackrel{f}{\to} (D,y) \big) & = \id \times R
    \big( (C,P(f)(y)) \stackrel{f}{\to} (D,y) \big).
  \end{align*}
  Thus there is a natural isomorphism $R \cong \hat{R}$, and 
  $F$ and $G$ form an equivalence.
  \qed
\end{proof}

\begin{definition}
  Define a presheaf $\Aut \in \PSh(\Cm)$ by
  \begin{align*}
    \Aut(C) & = \{ i \colon C \stackrel{\cong}{\to} C' \mid C' \in \C
    \}, \\ 
    \Aut\big( C \stackrel{k}{\rightarrowtail} D \big) \big( j \colon D
    \stackrel{\cong}{\to} D' \big) & = j \big|_{k(C)} \after k
    \colon C \stackrel{\cong}{\to} j(k(C)) =D'.
  \end{align*}
\end{definition}

Notice that $\Aut(C)$ contains the automorphism group
of $C$. Also, any automorphism of $A$ induces an element of
$\Aut(C)$.

The category $\int_{\Cm}\Aut$ of elements of $\Aut$ unfolds explicitly
to the following. Objects are pairs $(C,i)$ of 
$C \in \C$ and a $*$-isomorphism $i \colon C \smash{\stackrel{\cong}{\to}}
C'$. A morphism $(C,i) \to (D,j)$ is an injective $*$-homomorphism $k 
\colon C \rightarrowtail D$ such that $i = j \after k$.

\begin{proposition}
  The categories $\Cs$ and $\int_{\Cm}\Aut$ are equivalent. 
\end{proposition}
\begin{proof}
  Define a functor $F \colon \Cs \to \int_{\Cm}\Aut$ by
  $F(C)=(C, \id[C])$ on objects and $F(C \subseteq D) =
  (C \hookrightarrow D)$ on morphisms. Define a functor $G \colon
  \int_{\Cm}\Aut \to \Cs$ by $G(C,i) = i(C) = \cod(i)$ on objects and
  $G\big(k \colon (C,i) \to (D,j)\big) = (i(C) \subseteq j(D))$ on
  morphisms. Then $GF(C)=C$, and $FG(C,i) = (i(C), \id[i(C)]) \cong
  (C,i)$, so that $F$ and $G$ implement an equivalence. 
  \qed
\end{proof}

\begin{theorem}\label{CsCm}
  The toposes $\PSh(\Cs)$ and $\PSh(\Cm)/\Aut$ are equivalent.
\end{theorem}
\begin{proof}
  Combining the previous two lemmas, the equivalence is implemented
  explicitly by the functor $F \colon \PSh(\Cm)/\Aut \to
  \PSh(\Cs)$ defined by 
  \begin{align*}
    F\big(P \stackrel{\alpha}{\Rightarrow} \Aut \big)(C) & =
    \alpha_C^{-1}(\id[C]) \\
    F\big(P \stackrel{\alpha}{\Rightarrow} \Aut \big)(C \subseteq D) & =
    P(C \hookrightarrow D)
  \end{align*}
  and the functor $G \colon \PSh(\Cs) \to \PSh(\Cm)/\Aut$
  defined by $G(R) = \big( P \stackrel{\alpha}{\Rightarrow} \Aut \big)$,
  \begin{align*}
    P(C) & = \coprod_{i \colon C \stackrel{\cong}{\to} C'} R(i(C)), \\
    P\big(C \stackrel{k}{\rightarrowtail} D \big) & = \coprod_{j \colon D
      \stackrel{\cong}{\to} D'} R\big(j(k(C)) \subseteq j(D)\big), \\
    \alpha_C(\kappa_i(r)) & = i.
  \end{align*}
  This proves the theorem.
  \qed
\end{proof}

Hence the topos $\cat{T}=\PSh(\Cm)$ is an \emph{{\'e}tendue}, which means it is ``locally like a space''; more precisely, it contains an object $E$ such that the unique map from $E$ to the terminal object is an epimorphism and the slice category $\cat{T}/E$ is (equivalent to) a localic topos. In this case, the object $E$ is the presheaf $\Aut$.

\begin{lemma}\label{slices}
  If $F \colon \cat{C} \to \cat{D}$ is (half of) an equivalence, $X$
  is any object of $\cat{C}$ and $Y \cong F(X)$, then the slice
  categories $\cat{C}/X$ and $\cat{D}/Y$ are equivalent.
\end{lemma}
\begin{proof}
  Let $G \colon \cat{D} \to \cat{C}$ be the other half of the given
  equivalence, and pick an isomorphism $i \colon Y \to F(X)$. Define a
  functor $H \colon \cat{C}/X \to \cat{D}/Y$ by $H(a \colon A \to X) =
  (i \after Fa \colon FA \to Y)$ and $H(f \colon a \to b) = Ff$. Define a
  functor $K \colon \cat{D}/Y \to \cat{C}/X$ by $K(a \colon A \to Y) =
  (\eta_X^{-1} \after Gi \after Ga \colon GA \to X)$ and $K(f \colon a
  \to b) = Gf$. By naturality of $\eta^{-1}$ we then have $KH(a) \cong
  a$. And because $G\varepsilon = \eta^{-1}$ we also have $HK(a) \cong a$. 
  \qed
\end{proof}

\begin{lemma}\label{catvsPSh}
  If the categories $\cat{C}$ and $\cat{D}$ are equivalent, then
  the toposes $\PSh(\cat{C})$ and $\PSh(\Cat{D})$ are equivalent.
\end{lemma}
\begin{proof}
  Given functors $F \colon \cat{C} \to \cat{D}$ and $G \colon
  \cat{D} \to \cat{C}$ that form an equivalence, one directy verifies
  that $(-) \after G\colon\PSh(\cat{C}) \to \PSh(\cat{D})$ and $(-)
  \after F \colon\PSh(\cat{D}) \to \PSh(\cat{C})$ also form an equivalence. 
  \qed
\end{proof}


\begin{theorem}\label{CmCs}
  If $\Cm(A)$ and $\Cm(B)$ are equivalent categories, then $\Cs(A)$
  and $\Cs(B)$ are Morita-equivalent posets, \ie the toposes $\PSh(\Cs(A))$ and
  $\PSh(\Cs(B))$ are equivalent.
\end{theorem}
\begin{proof}
  If $\Cm(A) \simeq \Cm(B)$, then $\PSh(\Cm(A)) \simeq \PSh(\Cm(B))$ by
  Lemma~\ref{catvsPSh}. Moreover, the object $\Aut_B$ is (isomorphic
  to) the image of the object $\Aut_A$ under this equivalence.
  Hence
  \[
  \PSh(\Cs(A)) \simeq \PSh(\Cm(A))/\Aut_A \simeq \PSh(\Cm(B))/\Aut_B
  \simeq \PSh(\Cs(B))
  \]
  by Theorem~\ref{CsCm}.
  \qed
\end{proof}

\begin{remark}\label{invariants}
  Hence $\Cm(A)$ is an invariant of the topos $\PSh(\Cs(A))$ as well as of the C*-algebra
  $A$. It is not a complete invariant for the latter, however, as shown by Lemma~\ref{terminal}.
  For example, $\Cm(\mathbb{M}_n(\field{C})) \simeq \Cm(\field{C}^n)$, but $\Cs(\mathbb{M}_n(\field{C}))
  \not\cong \Cs(\field{C}^n)$, and certainly $\mathbb{M}_n(\field{C}) \not\cong
  \field{C}^n$.

  We have relied heavily on equivalences of categories, and indeed a
  logical formula holds in the topos $\PSh(\cat{C})$ if and only
  if it holds in $\PSh(\cat{D})$ for equivalent categories
  $\cat{C}$ and $\cat{D}$. Therefore one might argue that $\Cm$ has
  too many morphisms, as compared to $\Cs$, for toposes based on
  it to have internal logics that are interesting from the point of
  view of foundations of quantum mechanics. Instead of equivalences, one
  could consider isomorphisms of categories. This also resembles
  the original Mackey--Piron question more closely. After all, an equivalence of
  partial orders is automatically an isomorphism. The following
  theorem shows that $\Cm(A)$ is a weaker invariant of $A$ than
  $\Cs(A)$, in this sense.
\end{remark}

\begin{theorem}
  If $\Cm(A)$ and $\Cm(B)$ are isomorphic categories, then $\Cs(A)$
  and $\Cs(B)$ are isomorphic posets.
\end{theorem}
\begin{proof}
  Let $K \colon \Cm(A) \to \Cm(B)$ be an isomorphism. Suppose that
  $C,D \in \Cm(A)$ satisfy $C \subseteq D$. Consider the subcategory
  $\Cm(D)$ of $\Cm(A)$. On the one hand, by Lemma~\ref{charCmcomm} 
  it is isomorphic to $P(X) \rtimes S(X)$ for $X=\Spec(D)$, and therefore
  has a faithful retraction $F_A$ of the inclusion $\Cm(D) \to
  \Cm(D)(0,0)$ by Theorem~\ref{charPM}. On the other hand, $K$ maps it
  to $\Cm(K(D))$, which is isomorphic to $P(Y) \rtimes S(Y)$ for
  $Y=\Spec(K(D))$, and therefore similarly has a retraction
  $F_B$. Moreover, we have $KF_A = F_BK$. Now, by
  Theorem~\ref{charPM}, inclusions in $\Cm$ are characterized among
  all morphisms $f$ by $F(f)=1$. Hence $F_B(K(C \hookrightarrow D)) =
  KF_A(C \hookrightarrow  D)=K(1)=1$, and therefore $K(C) \subseteq K(D)$.
  \qed
\end{proof}

It remains open whether existence of an isomorphism $\Cs(A)
\cong \Cs(B)$ implies existence of an isomorphism $\Cm(A) \cong
\Cm(B)$. This question can be reduced as follows, at least in finite
dimension, because every
injective *-morphism factors uniquely as a $*$-isomorphism followed by
an inclusion. Write $\Ci(A)$ for the category with $\C(A)$ for objects
and $*$-isomorphisms as morphisms.
Supposing an isomorphism $F \colon \Cs(A) \to \Cs(B)$,
we have $\Cm(A) \cong \Cm(B)$ if and only if there is an isomorphism $G \colon
\Ci(A) \to \Ci(B)$ that coincides with $F$ on objects. Now, in case $A$ is
(isomorphic to) $\mathbb{M}_n(\field{C})$, (so is $B$, and) if $C,D \in
\C(A)$ are isomorphic then so are $F(C)$ and $F(D)$: if $C \cong
D$, then $\dim(C)=\dim(D)$, so $\dim(F(C))=\dim(F(D))$ because $F$
preserves maximal chains, and hence
$F(C) \cong F(D)$. However, it is not clear whether this behaviour is
functorial, \ie extends to a functor $G$, or generalizes to infinite dimension.

\appendix

\section{Inverse semigroups and {\'e}tendues}\label{sec:inversesemigroups}

The direct proof of Theorem~\ref{CsCm} follows from~\cite[A.1.1.7]{johnstone:elephant},
but it can also be arrived at through a detour via inverse semigroups, based
on results due to Funk~\cite{funk:semigroupsandtoposes}. This appendix
describes the latter intermediate results, which might be of
independent interest. For the rest of this appendix, we fix a unital C*-algebra $A$, and may therefore write $\Cs$ for $\Cs(A)$ and $\Cm$ for $\Cm(A)$.

\begin{definition}
  Define a set $T$ with functions $T \times T \stackrel{\cdot}{\to} T$
  and $T \stackrel{*}{\to} T$ by: 
  \begin{align*}
    & T = \left\{ C \stackrel{i}{\rightarrowtail} A
      \mid C \in \C, \; i \text{ is an injective
        $*$-homomorphism} \right\},\\
    & (C' \stackrel{i'}{\rightarrowtail} A) \cdot (C
    \stackrel{i}{\rightarrowtail} A) =
    (i^{-1}(C') \stackrel{i' \after i}{\rightarrowtail} A), \\
    & (C \stackrel{i}{\rightarrowtail} A)^* = (i(C)
    \stackrel{i^{-1}}{\rightarrowtail} A).
  \end{align*}
  The multiplication is well-defined, because the inverse image of a
  *-algebra under a $*$-homomorphism is again a *-algebra, and the inverse
  image of a closed set is again a closed set, so that $i^{-1}(C)$ is
  indeed a commutative C*-algebra.  The operation * is well-defined
  because of Lemma~\ref{cstarimage}; and on the image,
  $i^{-1}$ is a well-defined injective $*$-homomorphism. One can
  verify that together, these data form an inverse semigroup; that is,
  multiplication is associative, and $i^*$ is the unique element with
  $ii^*i=i$ and $i^*ii^*=i^*$.
\end{definition}

\begin{lemma}\label{domcodT}
  For $(C \stackrel{i}{\rightarrowtail} A) \in T$, we have $i^*i = (C
  \hookrightarrow A)$ and $ii^* = (i(C) \hookrightarrow A)$.  
\end{lemma}
\begin{proof}
  For the former claim:
  \[
  (C \stackrel{i}{\rightarrowtail} A)^* \cdot (C \stackrel{i}{\rightarrowtail} A) 
  = (i(C) \stackrel{i^{-1}}{\rightarrowtail} A) \cdot (C \stackrel{i}{\rightarrowtail} A) 
  = (i^{-1}(i(C)) \stackrel{i^{-1} \after i}{\rightarrowtail} A) 
  = (C \hookrightarrow A).
  \]
  For the latter claim:
  \begin{align*}
    (C \stackrel{i}{\rightarrowtail} A) \cdot (C \stackrel{i}{\rightarrowtail} A)^*
    & = (C \stackrel{i}{\rightarrowtail} A) \cdot (i(C) \stackrel{i^{-1}}{\rightarrowtail} A) \\
    & = ((i^{-1})^{-1}(C) \stackrel{i \after i^{-1}}{\rightarrowtail} A) 
    = (i(C) \hookrightarrow A).
  \end{align*}
  This proves the lemma.
  \qed
\end{proof}

\begin{definition}
  For any inverse semigroup $T$, one can define the groupoid $G(T)$ whose
  objects are the idempotents of $T$, i.e. the elements $e \in T$ with $e^2=e$. A morphism
  $e \to f$ is an element $t \in T$ satisfying $e=t^*t$ and $tt^*=f$.
\end{definition}

\begin{proposition}\label{GTA}
  The groupoids $G(T)$ and $\Ci$ are isomorphic.
\end{proposition}
\begin{proof}
  An element $(C \stackrel{i}{\rightarrowtail} A)$ of $T$ is idempotent when $i^{-1}(C)=C$
  and $i^2=i$ on $C$. That is, the objects of $G(T)$ are the
  inclusions $(C \hookrightarrow A)$ of commutative C*-subalgebras; we
  can identify them with $\C$.

  A morphism $C \to C'$ in $G(T)$ is an element $(D \stackrel{j}{\rightarrowtail} A)$ of $T$ 
  such that $(C \hookrightarrow A) = j^*j = (D \hookrightarrow A)$
  and $(C' \hookrightarrow A) = jj^* = (j(D) \hookrightarrow A)$, \ie
  $C=D$ and $C'=j(D)$. That is, a morphism $C \to C'$ is an injective
  $*$-homomorphism $j \colon C \rightarrowtail C'$ that satisfies
  $j(D)=C'$, \ie that is also surjective. In other words, a morphism
  $C \to C'$ is a $*$-isomorphism $C \to C'$.
  \qed
\end{proof}

\begin{definition}
  For any inverse semigroup $T$, one can define a partial order on the
  set $E(T)=\{ e \in T \mid e^2=e \}$ of idempotents by $e \leq f$ iff $e=fe$.
\end{definition}

In fact, $G(T)$ is an ordered groupoid, with $G(T)_0 = E(T)$.

\begin{proposition}\label{ETA}
  The posets $E(T)$ and $\Cs$ are isomorphic.
\end{proposition}
\begin{proof}
  As with $G(T)$, objects of $E(T)$ can be identified with $\C$. Moreover,
  there is an arrow $C \to C'$ if and only if 
  \[
  (C \hookrightarrow A) = (C' \hookrightarrow A) \cdot (C
  \hookrightarrow A) = (C \cap C' \hookrightarrow A),
  \]
  \ie when $C \cap C'=C$. That is, there is an arrow $C \to C'$ iff
  $C \subseteq C'$.
  \qed
\end{proof}

Also, $G(T)$ is always a subcategory of the following category $L(T)$.

\begin{definition}
  For any inverse semigroup $T$, one can define the left-cancellative
  category $L(T)$ whose objects are the idempotents of $T$. A morphism
  $e \to f$ is an element $t \in T$ satisfying $e=t^*t$ and $t=ft$.
\end{definition}

\begin{proposition}\label{LTA}
  The categories $L(T)$ and $\Cm$ are isomorphic.
\end{proposition}
\begin{proof}  
  As with $G(T)$, objects of $L(T)$ can be identified with $\C$.
  A morphism $C \to C'$ in $L(T)$ is an element
  $(j \colon D \rightarrowtail A)$ of $T$ such that $(C
  \hookrightarrow A) = j^*j = (D \hookrightarrow A)$ and
  \[
  (D \stackrel{j}{\rightarrowtail} A) = (C' \hookrightarrow A) \cdot
  (D \stackrel{j}{\rightarrowtail} A) = (j^{-1}(C')
  \stackrel{j}{\rightarrowtail} A).
  \]
  That is, a morphism $C \to C'$ is an injective $*$-homomorphism $j \colon
  C \rightarrowtail A$ such that $C=j^{-1}(C')$.
  Hence we can identify morphisms $C \to C'$ with injective
  $*$-homomorphisms $j \colon C \rightarrowtail C'$.  
  \qed
\end{proof}

Every ordered groupoid $G$ has a classifying topos $\B(G)$. We now
describe the topos $\B(G(T))$ explicitly, unfolding the definitions
on~\cite[page~487]{funk:semigroupsandtoposes}. 

For a presheaf $P \colon \Cs\op \to \Cat{Set}$, define another
presheaf $P^* \colon \Cs\op \to \Cat{Set}$ by 
\[
P^*(C) = \{ (j,x) \mid j \in \Ci(A)(C,C'), x \in
P(C') \}.
\]
On a morphism $C \subseteq D$, the presheaf $P^* \colon P^*(D) \to P^*(C)$ acts as
\[
(k \colon D' \stackrel{\cong}{\to} D, y \in P(D'))
\longmapsto
\big( k\big|_{C} \colon C \stackrel{\cong}{\to} k(C), \,P(k(C) \subseteq
D')(y) \big).
\]
An object of $\B(G(T))$ is a pair $(P,\theta)$ of a presheaf $P \colon
\Cs\op \to \Cat{Set}$ and a natural transformation $\theta \colon P^*
\Rightarrow P$. A morphism $(P,\theta) \to (Q,\xi)$ is a natural
transformation $\alpha \colon P \Rightarrow Q$ satisfying $\alpha \after
\theta = \xi \after \alpha^*$, where the natural transformation $\alpha^*
\colon P^* \Rightarrow Q^*$ is defined by $\alpha^*_C(j,x) = (j,\alpha_C(x))$.

\begin{lemma}\label{CmB}
  The toposes $\PSh(\Cm)$ and $\B(G(T))$ are equivalent.
\end{lemma}
\begin{proof}
  Combine Proposition~\ref{LTA} with
  \cite[Proposition~1.12]{funk:semigroupsandtoposes}. 
  Explicitly, $(P, \theta)$ in $\B(G(T))$ gets mapped to $F \colon
  \Cm(A)\op \to \Cat{Set}$ defined by $F(C)=P(C)$ and
  \[
  F(k \colon C \rightarrowtail D)(y) = \theta_C( k \colon C
  \stackrel{\cong}{\to} k(C), \, P(k(C) \subseteq D)(y) ).
  \]
  Conversely, $F$ in $\PSh(\Cm)$ gets mapped to
  $(P,\theta)$, where
  \begin{align*}
    P(C) & = F(C), \\
    P(C \subseteq D) & = F(C \hookrightarrow D), \\
    \theta_C (j \colon C \stackrel{\cong}{\to} C', x \in F(C')) & = F(C'
    \stackrel{j^{-1}}{\to} C \subseteq D) (x).
  \end{align*}
  This finishes the proof. 
  \qed
\end{proof}

There is a canonical object $\mathbf{S}=(S,\theta)$ in $\B(G(T))$,
defined as follows.
\begin{align*}
  S(C) & = \{ i \colon C \rightarrowtail A \}, \\
  S(C \subseteq D)(j \colon D \rightarrowtail A) & = (j\big|_C \colon
  C \rightarrowtail A).
\end{align*}
In this case $S^*$ becomes
\begin{align*}
  S^*(C) & = \{(j,i) \mid j \colon C \stackrel{\cong}{\to} C', i
  \colon C' \rightarrowtail A \}, \\
  S^*(C \subseteq D)(j,i) & = (j \mid_C \colon C \stackrel{\cong}{\to}
  j(C),\, i\big|_{j(C)} \colon j(C) \rightarrowtail A ).
\end{align*}
Hence we can define a natural transformation $\theta \colon S^*
\Rightarrow S$ by 
\[
\theta_C(j,i) = i \after j.
\]

The equivalence of the previous lemma maps $\mathbf{S}$ in $\B(G(T))$
to $\mathbf{D}$ in $\PSh(\Cm)$:
\begin{align*}
  \mathbf{D}(C) & = \{ i \colon C \rightarrowtail A \}, \\
  \mathbf{D}(k \colon C \rightarrowtail D)(j \colon D
  \rightarrowtail A) & = (j \after k \colon C \rightarrowtail A).
\end{align*}

Technically, the topos $\B(G(T))$ is an {\'e}tendue: the
unique morphism from some object $\mathbf{S}$ to the terminal object
is epic, and the slice topos $\B(G(T))/\mathbf{S}$ is (equivalent
to) a localic topos. The following lemma makes the latter equivalence
explicit.

\begin{lemma}
  The toposes $\B(G(T))/\mathbf{S}$ and $\PSh(\Cs)$ are equivalent. 
\end{lemma}
\begin{proof}
  Combine Proposition~\ref{ETA} with equation (1) in \cite[page
  488]{funk:semigroupsandtoposes}.
  \qed
\end{proof}

Combining the previous two lemmas, we find:

\begin{theorem}
  The toposes $\PSh(\Cm)/\mathbf{D}$ and $\PSh(\Cs)$ equivalent. 
  \qed
\end{theorem}

In our specific application, we have more information and it is
helpful to reformulate things slightly. By Lemma~\ref{cstarimage},
giving an injective $*$-homomorphism $i \colon C \rightarrowtail A$ is
the same as giving a $*$-isomorphism $C \cong C'$ for 
some $C' \in \C$ (by taking $C'=i(C)$). Hence $\mathbf{S}$ is
isomorphic to the object
$\mathbf{Aut}=(\mathrm{Aut},\theta)$ in $\B(G(T))$ with
$\theta_C(j,i)=i \after j$. This leads to Theorem~\ref{CsCm}.

\bibliographystyle{plain}
\bibliography{cccc}

\begin{thebibliography}{10}

\bibitem{vdbergheunen:nogo}
Benno van~den Berg and Chris Heunen.
\newblock No-go theorems for functorial localic spectra of noncommutative
  rings.
\newblock {\em Electronic Proceedings in Theoretical Computer Science},
  95:21--25, 2012.

\bibitem{vdbergheunen:colim}
Benno van~den Berg and Chris Heunen.
\newblock Noncommutativity as a colimit.
\newblock {\em Applied Categorical Structures}, 20(4):393--414, 2012.

\bibitem{blasssagan:mobius}
Andreas Blass and Bruce~E. Sagan.
\newblock M{\"o}bius functions of lattices.
\newblock {\em Advances in Mathematics}, 127:94--123, 1997.

\bibitem{bunge:presheaves}
Marta Bunge.
\newblock Internal presheaves toposes.
\newblock {\em Cahiers de topologie et g{\'e}om{\'e}trie diff{\'e}rentielle
  cat{\'e}goriques}, 18(3):291--330, 1977.

\bibitem{bures:masas}
Donald Bures.
\newblock {\em Abelian subalgebras of von {N}eumann algebras}.
\newblock Number 110 in Memoirs. American Mathematical Society, 1971.

\bibitem{davidson:cstar}
Kenneth~R. Davidson.
\newblock {\em C*-algebras by example}.
\newblock American Mathematical Society, 1991.

\bibitem{doeringisham:topos}
Andreas D{\"o}ring and Christopher~J. Isham.
\newblock {\em Deep Beauty}, chapter Topos methods in the foundations of
  physics.
\newblock Cambridge University Press, 2011.

\bibitem{firby:compactifications1}
Paul~A. Firby.
\newblock Lattices and compactifications {I}.
\newblock {\em Proceedings of the London Mathematical Society}, 27:22--50,
  1973.

\bibitem{firby:compactifications2}
Paul~A. Firby.
\newblock Lattices and compactifications {II}.
\newblock {\em Proceedings of the London Mathematical Society}, 27:51--60,
  1973.

\bibitem{funk:semigroupsandtoposes}
Jonathon Funk.
\newblock Semigroups and toposes.
\newblock {\em Semigroup forum}, 75:480--519, 2007.

\bibitem{hamhalter:pseudojordan}
Jan Hamhalter.
\newblock Isomorphisms of ordered structures of abelian {C*}-subalgebras of
  {C*}-algebras.
\newblock {\em Journal of Mathematical Analysis and Applications},
  383:391--399, 2011.

\bibitem{hamhalterturilova:pseudojordan}
Jan Hamhalter and Ekaterina Turilova.
\newblock Structure of associative subalgebras of {J}ordan operator algebras.
\newblock {\em Quarterly Journal of Mathematics}, 64(2):397--408, 2013.

\bibitem{hardingdoering:jordan}
John Harding and Andreas D{\"o}ring.
\newblock Abelian subalgebras and the {J}ordan structure of a von {N}eumann
  algebra.
\newblock {\em Houston Journal of Mathematics}, 2014.

\bibitem{heunen:complementarity}
Chris Heunen.
\newblock Complementarity in categorical quantum mechanics.
\newblock {\em Foundations of Physics}, 42(7):856--873, 2012.

\bibitem{heunenetal:topos}
Chris Heunen, Nicolaas~P. Landsman, and Bas Spitters.
\newblock A topos for algebraic quantum theory.
\newblock {\em Communications in Mathematical Physics}, 291:63--110, 2009.

\bibitem{heunenetal:bohrification}
Chris Heunen, Nicolaas~P. Landsman, and Bas Spitters.
\newblock {\em Deep Beauty}, chapter Bohrification.
\newblock Cambridge University Press, 2011.

\bibitem{heunenreyes:diagonal}
Chris Heunen and Manuel~L. Reyes.
\newblock Diagonalizing matrices over {AW}*-algebras.
\newblock {\em Journal of Functional Analysis}, 264(8):1873--1898, 2013.

\bibitem{heunenreyes:awstar}
Chris Heunen and Manuel~L. Reyes.
\newblock Active lattices determine {AW}*-algebras.
\newblock {\em Journal of Mathematical Analysis and Applications}, 2014.

\bibitem{johnstone:elephant}
Peter~T. Johnstone.
\newblock {\em Sketches of an elephant: A topos theory compendium}.
\newblock Oxford University Press, 2002.

\bibitem{kadisonringrose:operatoralgebras}
Richard~V. Kadison and John~R. Ringrose.
\newblock {\em Fundamentals of the theory of operator algebras}.
\newblock Academic Press, 1983.

\bibitem{kalmbach:orthomodularlattices}
Gudrun Kalmbach.
\newblock {\em Orthomodular Lattices}.
\newblock Academic Press, 1983.

\bibitem{kalmbach:measures}
Gudrun Kalmbach.
\newblock {\em Measures and {H}ilbert lattices}.
\newblock World Scientific, 1986.

\bibitem{maclanemoerdijk:sheaves}
Saunders {Mac Lane} and Ieke Moerdijk.
\newblock {\em Sheaves in geometry and logic}.
\newblock Springer, 1992.

\bibitem{piron:foundations}
Constantin Piron.
\newblock {\em Foundations of quantum physics}.
\newblock Number~19 in Mathematical Physics Monographs. W.A. Benjamin, 1976.

\bibitem{redei:quantumlogic}
Mikl{\'o}s R{\'e}dei.
\newblock {\em Quantum Logic in Algebraic Approach}.
\newblock Kluwer, 1998.

\bibitem{segal:decomposition}
Irving Segal.
\newblock {\em Decompositions of operator algebras {II}: multiplicity theory}.
\newblock Number~9 in Memoirs. American Mathematical Society, 1951.

\bibitem{sinclairsmith:masas}
Allan~M. Sinclair and Roger~R. Smith.
\newblock {\em Finite von {N}eumann algebras and masas}.
\newblock Number 351 in London Mathematical Society lecture notes. Cambridge
  University Press, 2008.

\bibitem{soler:orthomodular}
Maria~Pia Sol{\`e}r.
\newblock Characterization of {H}ilbert spaces by orthomodular spaces.
\newblock {\em Communications in Algebra}, 23(1):219--243, 1995.

\bibitem{tomiyama:masas}
Jun Tomiyama.
\newblock On some types of maximal abelian subalgebras.
\newblock {\em Journal of functional analysis}, 10(373--386), 1972.

\bibitem{wilce:testspaces}
Alexander Wilce.
\newblock {\em Handbook of quantum logic}, volume~II, chapter Test spaces.
\newblock Elsevier, 2008.

\bibitem{yoon:partitionlattices}
Young-Jin Yoon.
\newblock Characterizations of partition lattices.
\newblock {\em Bulletin of the Korean Mathematical Society}, 31(2):237--242,
  1994.

\end{thebibliography}

\end{document}